\documentclass[11pt, letterpaper,reqno]{amsart}

\usepackage{amssymb}
\usepackage{graphicx}
\usepackage[hyphens]{url}
\usepackage{hyperref}
\usepackage{appendix}
\usepackage{verbatim}

\numberwithin{equation}{section}

\newtheorem{theorem}{Theorem}[section]
\newtheorem{assumption}{Assumption}[section]
\newtheorem{corollary}{Corollary}[section]
\newtheorem{lemma}{Lemma}[section]
\newtheorem{proposition}{Proposition}[section]
\theoremstyle{definition}
\newtheorem{definition}{Definition}[section]
\theoremstyle{remark}
\newtheorem{remark}{Remark}[section]
\newtheorem{example}{Example}[section]

\newcommand{\leref}{Lemma~\ref}

\newcommand{\coref}{Corollary~\ref}
\newcommand{\exref}{Example~\ref}

\newcommand{\thref}{Theorem~\ref}

\newcommand{\asref}{Assumption~\ref}

\newcommand{\Q}{\mathbb{Q}}
\newcommand{\R}{\mathbb{R}}

\newcommand{\T}{\mathcal{T}}

\newcommand{\cP}{\mathcal{P}}
\newcommand{\cQ}{\mathcal{Q}}

\newcommand{\cH}{\mathcal{H}}

\newcommand{\fQ}{\mathfrak{Q}}

\newcommand{\eps}{\varepsilon}

\newcommand{\f}{\mathfrak{f}}

\newcommand{\Mapsto}{\twoheadrightarrow}

\title[]{On Arbitrage and Duality under Model Uncertainty and Portfolio  Constraints} \thanks{This version: January 12, 2015; First version: February 11, 2014.}
\author[]{Erhan Bayraktar} \thanks{This research is supported in part by the National Science Foundation under grant DMS-0955463.}  
\address{Department of Mathematics, University of Michigan}
\email{erhan@umich.edu}
\author[]{Zhou Zhou}
\address{Department of Mathematics, University of Michigan}
\email{zhouzhou@umich.edu}
\date{}
\begin{document}
\maketitle

\begin{abstract}
We consider the fundamental theorem of asset pricing (FTAP) and the hedging prices of options under non-dominated model uncertainty and portfolio constrains in discrete time. We first show that no arbitrage holds if and only if there exists some family of probability measures such that any admissible portfolio value process is a local super-martingale under these measures. We also get the non-dominated optional decomposition with constraints. From this decomposition, we obtain the duality of the super-hedging prices of European options, as well as the sub- and super-hedging prices of  American options. Finally, we get the FTAP and the duality of super-hedging prices in a market where stocks are traded dynamically and options are traded statically.
\end{abstract}

\section{Introduction}
We consider the fundamental theorem of asset pricing (FTAP) and the hedging prices of European and American options under the non-dominated model certainty framework of \cite{Nutz2} with \emph{convex closed portfolio constraints} in discrete time. 
%
We first show that no arbitrage in the quasi-sure sense is equivalent to the existence of a set of probability measures; under each of these measures any admissible portfolio value process is a local super-martingale. Then we get the non-dominated version the of optional decomposition under portfolio constraints. From this optional decomposition, we get the duality of super- and sub-hedging prices of European and American options. We also show that the optimal super-hedging strategies exist. Finally, we add options to the market and get the FTAP and the duality of super-hedging prices of European options by using semi-static trading strategies (i.e., strategies dynamically trading in stocks and statically trading in options).

Our results generalize the ones in \cite[Section 9]{SF} to a non-dominated model-uncertainty set-up, and extend the results in \cite{Nutz2} to the case where portfolio constraints are involved. These conclusions are general enough to cover many interesting models with the so-called \emph{delta constraints}; for example, when shorting stocks is not allowed, or some stocks enter or leave the market at certain times. 

Compared to \cite[Section 9]{SF}, the main difficulty in our setting is due to the fact that the set of probability measures does not admit a dominating measure. We use the measurable selection mechanism developed in \cite{Nutz2} to overcome this difficulty, i.e., first establish the FTAP and super-hedging result in one period, and then \lq\lq measurably\rq\rq\ glue all the periods together to derive multiple-period versions. It is therefore  of crucial importance to get the one-period results. In \cite{Nutz2}, Lemma 3.3 serves as a fundamental tool to show the FTAP and super-hedging result a in one-period model, whose proof relies on an induction on the number of stocks and a separating hyperplane argument. While in our set-up, both the induction and separating argument do not work due to the presence of constraints. In this paper, we instead use a finite covering argument to overcome the difficulty stemming from constraints. Another major difference from \cite{Nutz2} is the proof for the existence of optimal super-hedging strategy in multiple period, which is a direct consequence of Theorem 2.2 there. A key step in the proof of Theorem 2.2 is modifying the trading strategy to the one with fewer \lq\lq rank\rq\rq\ yet still giving the same portfolio value. However, this approach fails to work in our set-up, since the modification may not be admissible anymore due to the portfolio constrains. In our paper, we first find the optimal static trading strategy of options, and then find the optimal dynamical trading strategy of stocks by optional decomposition with constraints. Optional decomposition also helps us obtain the duality results for the American options.


We work within the no-arbitrage framework of \cite{Nutz2}, in which there is said to be an arbitrage when there exists a trading strategy whose gain is quasi surely non-negative and strictly positive with positive probability under an admissible measure. In this framework we are given a model and the non-dominated set of probability measures comes from estimating the parameters of the model. Since estimating results in confidence intervals for the parameters we end up with a set of non-dominated probability measures.

There is another no-arbitrage framework which was introduced by Acciaio et al. \cite{Schachermayer2}. In that framework, there is said to be an arbitrage if the gain from trading is strictly positive for all scenarios.  Under the framework of \cite{Schachermayer2}, the model uncertainty is in fact part of the model itself and the user of that model does not have confidence in her ability to estimate the parameters.  The choice between the framework of \cite{Schachermayer2} and the framework of \cite{Nutz2} is a modeling issue.

Our assumptions mainly contain two parts: (1) the closedness and convexity of the related control sets (see Assumptions \ref{a1}, \ref{a2}, \ref{a3} and \ref{a4}), and (2) some measurability assumptions (see the set-up of Section 3.1 and Assumptions \ref{a2} and \ref{a4}). The first part is almost necessary (see \exref{ex2.2}), and can be easily verified in many interesting cases (see e.g., \exref{ex2.1}). The second part is the analyticity of some relevant sets, which we make in order to apply measurable selection results and perform dynamic programming principle type arguments. Analyticity (which is a measurability concept more general than Borel measurability, so in particular every Borel set is analytic) is a minimal assumption one can have in order to have a dynamic programming principle and this goes well back to Blackwell. These concepts are covered by standard textbooks on measure theory, see e.g. \cite{MR3098996}. See also \cite{Shreve} for applications in stochastic control theory and the references therein. In Section 3.3, we provide some general and easily verifiable sufficient conditions for Assumptions \ref{a2}(iii) and \ref{a4}(ii), as well as Examples \ref{ex:straex} and \ref{ex3.2}.

The rest of the paper is organized as follows: We show the FTAP in one period and in multiple periods in Sections 2 and 3 respectively. In Section 4, we get the super-hedging result in one period. In Section 5, we provide the non-dominated optional decomposition with constraints in multiple periods. Then starting from the optional decomposition, we analyze the sub- and super-hedging prices of European and American options in multiple periods in Section 6. In Section 7, we add options to the market, and study the FTAP and super-hedging using semi-static trading strategies in multiple periods.  Finally in the appendix, we provide the proofs of Lemmas \ref{l6}, \ref{l15}, \ref{l9} and \ref{l7}; these proofs are with a lot of technicalities and can be safely skipped in the first reading. 

We devote the rest of this section to frequently used notation and concepts in the paper. 


\subsection{Frequently used notation and concepts}\label{sec:set-up}\label{notation}
\begin{itemize}
\item $\mathfrak{P}(\Omega)$ denotes set of all the probability measures on $(\Omega,\mathcal{B}(\Omega))$, where $\Omega$ is some polish space, and $\mathcal{B}(\Omega)$ denotes its Borel $\sigma$-algebra. $\mathfrak{P}(\Omega)$ is endowed with the topology of weak convergence.
\item $\Delta S_t(\omega,\cdot)=S_{t+1}(\omega,\cdot)-S_t(\omega),\ \omega\in\Omega_t:=\Omega^t$ (t-fold Cartesian product of $\Omega$). We may simply write $\Delta S$ when there is only one period (i.e., $t=0$).
\item Let $\cP\subset\mathfrak{P}(\Omega_t)$. A property holds $\cP-q.s.$ if and only if it holds $P$-a.s. for any $P\in\cP$. A set $A\in\Omega_t$ is $\cP$-polar if $\sup_{P\in\cP}P(A)=0$.
\item Let $\cP\subset\mathfrak{P}(\Omega)$. $\text{supp}_\cP(\Delta S)$ is defined as the smallest closed subset $A\subset\mathbb{R}^d$ such that $\Delta S\in A\ \cP-q.s.$. Define $N(\mathcal{P}):=\{H\in\mathbb{R}^d:\ H\Delta S=0,\ \mathcal{P}-q.s.\}$ and $N^\perp(\mathcal{P}):=\text{span}(\text{supp}_\cP(\Delta S))\subset\mathbb{R}^d$. Then $N^\perp(\mathcal{P})=(N(\mathcal{P}))^\perp$ by \cite[Lemma 2.6]{Nutz4}. Denote $N(P)=N(\{P\})$ and $N^\perp(P)=N^\perp(\{P\})$.
\item For $\mathcal{H}\subset\mathbb{R}^d$, $\mathcal{H}(\mathcal{P}):=\{H:\ H\in\text{proj}_{N^\perp(\mathcal{P})}(\mathcal{H})\}$. Denote $\mathcal{H}(P)=\mathcal{H}(\{P\})$.
\item For $\mathcal{H}\subset\mathbb{R}^d$, $\mathcal{C}_\mathcal{H}(\mathcal{P}):=\{c H:\ H\in\mathcal{H}(\mathcal{P}),\ c\geq 0\}$.  Denote $\mathcal{C}_\mathcal{H}(P)=\mathcal{C}_\mathcal{H}(\{P\})$.
\item $\mathfrak{C}_\mathcal{H}:=\{cH\in\mathbb{R}^d:\ H\in\mathcal{H},\ c\geq 0\}$, where $\mathcal{H}\subset\mathbb{R}^d$.
\item $(H\cdot S)_t=\sum_{i=0}^{t-1}H_i(S_{i+1}-S_i)$.
\item $\mathbb{R}^*:=[-\infty,\infty]$.
\item $||\cdot||$ represents the Euclidean norm.
\item $E_P|X|:=E_P|X^+|-E_P|X^-|$, and by convention $\infty-\infty=-\infty$. Similarly the conditional expectation is also defined in this extended sense. 
\item $L_+^0(\mathcal{P})$ is the space of random variables  $X$ on the corresponding topological space satisfying $X\geq 0\ \mathcal{P}-q.s.$, and $L^1(\mathcal{P})$ is the space of  random variables $X$ satisfying $\sup_{P\in\cP}E_P|X|<\infty$. Denote $L_+^0(P)=L_+^0(\{P\})$, and $L^1(P)=L^1(\{P\})$. Similar definitions holds for $L^0,\ L_+^1$ and $L^\infty$. We shall sometimes omit $\mathcal{P}$ or $P$ in $L_+^0,\ L^1$, etc., when there is no ambiguity. 
\item We say NA$(\mathcal{P})$ holds, if for any $H\in\mathcal{H}$ satisfying $(H\cdot S)_T\geq 0,\ \mathcal{P}-q.s.$, then $(H\cdot S)_T=0,\ \mathcal{P}-q.s.$, where $\mathcal{H}$ is some admissible control set of trading strategies for stocks. Denote NA$(P)$ for NA$(\{P\})$.
\item We write $Q\lll\mathcal{P}$, if there exists some $P\in\mathcal{P}$ such that $Q\ll P$.
\item Let $(X,\mathcal{G})$ be a measurable space and $Y$ be a topological space. A mapping $\Phi$ from $X$ to the power set of $Y$ is denoted by $\Phi:\ X\Mapsto Y$. We say $\Phi$ is measurable (resp. Borel measurable), if
\begin{equation}\label{e15}
\{x\in X:\ \Phi(x)\cap A\neq\emptyset\}\in\mathcal{G},\ \forall\text{ closed (resp. Borel measurable)}\ A\subset Y.
\end{equation}
$\Phi$ is closed (resp. compact) valued if $\Phi(x)\subset Y$ is closed (resp. compact) for all $x\in X$. We refer to \cite[Chapter 18]{IDA} for these concepts.
\item A set of random variables $A$ is $\mathcal{P}-q.s.$ closed, if $(a_n)_n\subset A$ convergent to some $a\ \mathcal{P}-q.s.$ implies $a\in A$.
\item For $\Phi:\ X\Mapsto Y$, Gr$(\Phi):=\{(x,y)\in X\times Y:\ y\in\Phi(x)\}$.
\item Let $X$ be a Polish space.  A set $A\subset X$ is analytic if it is the image of a Borel subset of another Polish space under a Borel measurable mapping. A function $f: X\mapsto\mathbb{R}^*$ is upper (resp. lower) semianalytic if the set $\{f>c\}$ (resp. $\{f<c\}$) is analytic. \lq\lq u.s.a.\rq\rq\ (resp. \lq\lq l.s.a.\rq\rq) is short for upper (resp. lower) semianalytic.
\item Let $X$ be a polish space. The $\sigma$-algebra $\cap_{P\in\mathfrak{P}(X)}\mathcal{B}(X)^P$ is called the universal completion of $\mathcal{B}(X)$, where $\mathcal{B}(X)^P$ is the $P$-completion of $\mathcal{B}(X)$. A set $A\subset X$ is universally measurable if $A\in\cap_{P\in\mathfrak{P}(X)}\mathcal{B}(X)^P$. A function $f$ is universally measurable if $f\in\cap_{P\in\mathfrak{P}(X)}\mathcal{B}(X)^P$. \lq\lq u.m.\rq\rq\ is short for universally measurable.
\item Let $X$ and $Y$ be some Borel spaces and $U:\ X\Mapsto Y$. Then $u$ is a u.m. selector of $U$, if $u:\ X\mapsto Y$ is u.m. and $u(\cdot)\in U(\cdot)$ on $\{ U\neq\emptyset\}$.
\end{itemize}

\section{The FTAP in one period}

We derive the FTAP for one-period model in this section. \thref{t1} is the main result of this section.

\subsection{The set-up and the main result}
Let $\mathcal{P}$ be a set of probability measures on a Polish space $\Omega$, which is assumed to be convex. Let $S_0\in\R^d$ be the initial stock price, and Borel measurable $S_1:\ \Omega\mapsto\mathbb{R}^d$ be the stock price at time $t=1$. Denote $\Delta S=S_1-S_0$. Let $\mathcal{H}\subset\mathbb{R}^d$ be the set of admissible trading strategies. We assume $\mathcal{H}$ satisfies the following conditions:
\begin{assumption}\label{a1}
 $\mathcal{C}_\mathcal{H}(\mathcal{P})$ is (i) convex, and (ii) closed.
\end{assumption}
\begin{example}\label{ex2.1}
Let  $\cH:=\prod_{i=1}^d[\underline a^i,\overline a^i]$ for some $\underline a^i,\overline a^i\in\R$ with $\underline a^i\leq\overline a^i,\ i=1\dotso,d$.
Then $\cH$ satisfies \asref{a1} for any $\cP\subset\mathfrak{P}(\Omega)$. Indeed, $\cH\subset\R^d$ is a bounded, closed, convex set with finitely many vertices, and so is $\cH(\cP)$. Hence the generated cone $\mathcal{C}_\cH(\cP)$ is convex and closed.
\end{example}
Define 
$$\mathcal{Q}:=\{Q\in\mathfrak{P}(\Omega):\ Q\lll\mathcal{P},\ E_Q|\Delta S|<\infty\text{ and } E_Q[H\Delta S]\leq 0,\ \forall H\in\mathcal{H} \}.$$
The following is the main result of this section:
\begin{theorem}\label{t1}
Let \asref{a1} hold. Then NA$(\mathcal{P})$ holds if and only if for any $P\in\mathcal{P}$, there exists $Q\in\mathcal{Q}$ dominating $P$.
\end{theorem}

\subsection{Proof for \thref{t1}}
Let us first prove the following lemma, which is the simplified version of \thref{t1} when $\mathcal{P}$ consists of a single probability measure.
\begin{lemma}\label{l1}
Let $P\in\mathfrak{P}(\Omega)$ and \asref{a1} w.r.t. $\mathcal{C}_\mathcal{H}(P)$ hold. Then $NA(P)$ holds if and only if  there exists $Q\sim P$, such that $E_Q|\Delta S|<\infty$ and $E_Q[H\Delta S]\leq 0$, for any $H\in\mathcal{H}$.
\end{lemma}
\begin{proof}
Sufficiency is obvious. We shall prove the necessity in two steps. W.l.o.g. we assume that $E_P|\Delta S|<\infty$ (see e.g., \cite[Lemma 3.2]{Nutz2}).\\
\textbf{Step 1:} In this step, we will show that $K-L_+^0$ is closed in $L^0$, where
$$K:=\{H\Delta S:\ H\in\mathcal{C}_\mathcal{H}(P)\}.$$
 Let $X_n=H_n\Delta S-Y_n\stackrel{P}{\rightarrow}X$, where $H_n\in\mathcal{C}_\mathcal{H}(P)$ and $Y_n\geq 0$. Without loss of generality, assume $X_n\rightarrow X,\ P$-a.s..  If $(H_n)_n$ is not bounded, then let $0<||H_{n_k}||\rightarrow\infty$ and we have that
$$\frac{H_{n_k}}{||H_{n_k}||}\Delta S=\frac{X_{n_k}}{||H_{n_k}||}+\frac{Y_{n_k}}{||H_{n_k}||}\geq\frac{X_{n_k}}{||H_{n_k}||}.$$
Taking limit on both sides along a further sub-sequence, we obtain that $H\Delta S\geq 0\ P$-a.s. for some $H\in\mathbb{R}^d$ with $||H||=1$. Since $\mathcal{C}_\mathcal{H}(P)$ is closed, $H\Delta S\in\mathcal{C}_\mathcal{H}(P)$. By NA$(P)$, $H\Delta S=0\ P$-a.s., which implies $H\in N(P)\cap N^\perp(P)=\{0\}$. This contradicts $||H||=1$. Therefore, $(H_n)_n$ is bounded, and thus there exists a subsequence $(H_{n_j})_j$ convergent to some $H'\in\mathcal{C}_\mathcal{H}(P)$. Then
$$0\leq Y_{n_j}=H_{n_j}\Delta S-X_{n_j}\rightarrow H'\Delta S-X=:Y,\quad P\text{-a.s..}$$
Then $X=H'\Delta S-Y\in K-L_+^0$.\\
\textbf{Step 2:} From Step 1, we know that $K':=(K-L_+^0)\cap L^1$ is a closed, convex cone in $L^1$, and contains $-L_+^\infty$. Also by NA$(P)$, $K'\cap L_+^1=\{0\}$. Then Kreps-Yan theorem (see e.g., \cite[Theorem 1.61]{SF}) implies the existence of $Q\sim P$ with $dQ/dP\in L_+^\infty(P)$, such that $E_Q[H\Delta S]\leq 0$ for any $H\in\mathcal{H}$.
\end{proof}
\begin{remark}
The FTAP under a single probability measure with constraints is analyzed in \cite[Chapter 9]{SF}. However, although the idea is quite insightful, the result there is not correct: what we need is the closedness of the generated cone $\mathcal{C}_\cH(P)$ instead of the closedness of $\mathcal{H}(P)$. (In this sense, our result is different from \cite{CW}; in \cite{CW} it is the closedness of the corresponding projection that matters.) Below is a counter-example to \cite[Theorem 9.9]{SF}.
\end{remark}
\begin{example}\label{ex2.2}
Consider the one-period model: there are two stocks $S^1$ and $S^2$ with the path space $\{(1,1)\}\times\{(s,0):\ s\in[1,2]\}$; let 
$$\mathcal{H}:=\{(h_1,h_2):\ h_1^2+(h_2-1)^2\leq 1\}.$$
be the set of admissible trading strategies; let $P$ be a probability measure on this path space such that $S_1^1$ is uniformly distributed on $[1,2]$.  It is easy to see that NA$(P)$ holds, and $\mathcal{H}$ satisfies the assumptions (a), (b) and (c) on \cite[page 350]{SF}. Let $H=(h_1,h_2)$ such that $H\Delta S=0,\ P$-a.s. Then $h_1(S_1^1-1)=h_2,\ P$-a.s., which implies $h_1=h_2=0$. By \cite[Remark 9.1]{SF}, $\mathcal{H}$ also satisfies assumption (d) on \cite[page 350]{SF}. 

Now suppose \cite[Theorem 9.9]{SF} holds, then there exists $Q\sim P$, such that 
\begin{equation}\label{e1}
E_Q[H\Delta S]\leq 0,\quad\forall H\in\mathcal{H}.
\end{equation}
Since $Q\sim P,\ E_Q(S_1^1-1)>0$. Take $(h_1,h_2)\in\mathcal{H}$ with $h_1,h_2>0$ and $h_2/h_1<E_Q(S_1^1-1)$. Then 
$$h_1E_Q(S_1^1-1)- h_2>0,$$
which contradicts \eqref{e1}.

In fact, it is not hard to see that in this example,
$$\mathcal{C}_\mathcal{H}(P)=\{(h_1,h_2): h_2>0 \text{ or } h_1=h_2=0\}$$
is not closed.
\end{example}

\begin{lemma}\label{l12}
Let \asref{a1}(ii) hold. Then there exists $P''\in\mathcal{P}$, such that $N^\perp(P'')=N^\perp(\mathcal{P})$ and NA$(P'')$ holds. 
\end{lemma}
\begin{proof}
Denote $\mathbb{H}:=\{H\in\mathcal{C}_\mathcal{H}(\mathcal{P}):\ ||H||=1\}.$ For any $H\in\mathbb{H}\subset N^\perp(\mathcal{P})$, by $NA(\mathcal{P})$, there exists $P_H\in\mathcal{P}$, such that $P_H(H\Delta S<0)>0$. It can be further shown that there exists $\eps_H>0$, such that for any $H'\in B(H,\eps_H)$, 
\begin{equation}\label{e28}
P_H(H'\Delta S<0)>0,
\end{equation}
where $B(H,\eps_H):=\{H''\in\mathbb{R}^d:\ ||H''-H||<\eps_H\}$. Indeed, there exists some $\delta>0$ such that
$P_H(H\Delta S<-\delta)>0.$
Then there exists some $M>0$, such that
$P_H(H\Delta S<-\delta,\ ||\Delta S||<M)>0.$
Taking $\eps_H:=\delta/M$, we have that for any $H'\in B(H,\eps_H)$, $P_H(H'\Delta S<0,\ ||\Delta S||<M)>0,$
which implies \eqref{e28}.

Because $\mathbb{H}\subset\cup_{H\in\mathbb{H}}B(H,\eps_H)$ and $\mathbb{H}$ is compact from \asref{a1}, there exists a finite cover of $\mathbb{H}$, i.e., $\mathbb{H}\subset\cup_{i=1}^nB(H_i,\eps_{H_i})$. Let $P'=\sum_{i=1}^na_iP_{H_i}$, with $\sum_{i=1}^na_i=1$ and $a_i>0,\ i=1,\dotso,n$. Then $P'\in\mathcal{P}$, and $P'(H\Delta S<0)>0$ for any $H\in\mathbb{H}$. 

Obviously, $N^\perp(P')\subset N^\perp(\mathcal{P})$.  If $N^\perp(P')=N^\perp(\mathcal{P})$, then let $P''=P'$. Otherwise, take $H\in N^\perp(\mathcal{P})\cap N(P')$. Then there exists $R_1\in\mathcal{P}$, such that $R_1(H\Delta S\neq 0)>0$. Let $R_1'=(P'+R_1)/2$. Then $P'\ll R_1'\in\mathcal{P}$, and thus $N^\perp(R_1')\supset N^\perp(P')$. Since $H\in N(P')\setminus N(R_1')$, we have that $N^\perp(R_1')\supsetneqq N^\perp(P')$. If $N^\perp(R_1')\subsetneqq N^\perp(\mathcal{P})$, then we can similarly construct $R_2'\in\mathcal{P}$, such that $R_2'\gg R_1'$ and $N^\perp(R_2')\supsetneqq N^\perp(R_1')$. Since $N^\perp(\mathcal{P})$ is a finite dimensional vector space, after finite such steps, we can find such $P''\in\mathcal{P}$ dominating $P'$ with $N^\perp(P'')=N^\perp(\mathcal{P})$. For any $H\in\mathbb{H}$, $P''(H\Delta S<0)>0$ since $P''\gg P'$. This implies that NA$(P'')$ holds.
\end{proof}

\begin{proof}[\textbf{Proof of \thref{t1}}]
\textit{Sufficiency.} If not, there exists $H\in\mathcal{H}$ and $P\in\mathcal{P}$, such that $H\Delta S\geq 0,\ P-a.s.$ and $P(H\Delta S>0)>0$. Take $Q\in\mathcal{Q}$ with $Q\gg P$. Then $E_Q[H\Delta S]\leq 0$, which contradicts $H\Delta S\geq 0\ Q-a.s.$ and $Q(H\Delta S>0)>0$.\\
\textit{Necessity.} Take $P\in\mathcal{P}$. By \leref{l12} there exists $P''\in\cP$ such that $N^\perp(P'')=N^\perp(\mathcal{P})$ and NA$(P'')$ holds. Let $R:=(P+P'')/2\in\cP$. Then $N^\perp(R)=N^\perp(P'')=N^\perp(\mathcal{P})$, and thus $\mathcal{C}_\mathcal{H}(R)=\mathcal{C}_\mathcal{H}(\mathcal{P})$ which is convex and closed by \asref{a1}. Besides, NA$(P'')$ implies that for any $H\in\mathcal{C}_\mathcal{H}(R)\setminus\{0\}=\mathcal{C}_\mathcal{H}(P'')\setminus\{0\}$, $P''(H\Delta S<0)>0$, and thus $R(H\Delta S<0)>0$ since $R\gg P''$. This shows that NA$(R)$ holds. From \leref{l1}, there exists $Q\sim R\gg P$, such that  $E_Q|\Delta S|<\infty$ and $E_Q[H\Delta S]\leq 0$ for any $H\in\mathcal{H}$.
\end{proof} 

\section{The FTAP in multiple periods}
We derive the FTAP in multiple periods in this section, and \thref{t2} is our main result. We will reduce it to a one-step problem and apply Theorem~\ref{t1}. 

\subsection{The set-up and the main result}
We use the set-up in \cite{Nutz2}. Let $T\in\mathbb{N}$ be the time Horizon and let $\Omega$ be a Polish space. For $t\in\{0,1,\dotso,T\}$, let $\Omega_t:=\Omega^t$ be the $t$-fold Cartesian product, with the convention that $\Omega_0$ is a singleton. We denote by $\mathcal{F}_t$ the universal completion of $\mathcal{B}(\Omega_t)$, and we shall often treat $\Omega_t$ as a subspace of $\Omega_T$. For each $t\in\{0,\dotso,T-1\}$ and $\omega\in\Omega_t$, we are given a nonempty convex set $\mathcal{P}_t(\omega)\subset\mathfrak{P}(\Omega)$ of probability measures.  Here $\mathcal{P}_t$ represents the possible models for the $t$-th period, given state $\omega$ at time $t$. We assume that for each $t$, the graph of $\mathcal{P}_t$ is analytic, which ensures by the Jankov-von Neumann Theorem (see, e.g., \cite[Proposition 7.49]{Shreve}) that $\mathcal{P}_t$ admits a u.m. selector, i.e., a u.m. kernel $P_t:\ \Omega_t\rightarrow \mathfrak{P}(\Omega)$ such that $P_t(\omega)\in\mathcal{P}_t(\omega)$ for all $\omega\in\Omega_t$. Let
\begin{equation}\notag
\mathcal{P}:=\{P_0\otimes\dotso\otimes P_{T-1}:\ P_t(\cdot)\in\mathcal{P}_t(\cdot),\ t=0,\dotso,T-1\},
\end{equation}
where each $P_t$ is a u.m. selector of $\mathcal{P}_t$, and
$$P_0\otimes\dotso\otimes P_{T-1}(A)=\int_{\Omega}\dotso\int_{\Omega} 1_A(\omega_1,\dotso,\omega_T)P_{T-1}(\omega_1,\dotso,\omega_{T-1};d\omega_T)\dotso P_0(d\omega_1),\ \ \ A\in\Omega_T.$$ \footnotetext{In order not to burden the reader with further notation we prefer use the same notation $\mathcal{P}$ for the set of probability measures in one-period models and multi-period models. We will do the same for other sets of probability measures that appear later in the paper and also for the set of admissible strategies.}

Let $S_t=(S_t^1,\dotso,S_t^d):\Omega_t\rightarrow\mathbb{R}^d$ be Borel measurable, which represents the price at time $t$ of a stock $S$ that can be traded dynamically in the market.

For each $t\in\{0,\dotso,T-1\}$ and $\omega\in\Omega_t$, we are given a set $\mathcal{H}_t(\omega)\subset\mathbb{R}^d$, which is thought as the set of admissible controls for the $t$-th period, given state $\omega$ at time $t$. We assume for each $t$,  graph$(\mathcal{H}_t)$ is analytic, and thus admits a u.m. selector; that is, an $\mathcal{F}_t$-measurable function $H_t(\cdot):\ \Omega_t\mapsto\mathbb{R}^d$, such that $H_t(\omega)\in\mathcal{H}_t(\omega)$. We introduce the set of admissible portfolio controls $\mathcal{H}$:
$$\mathcal{H}:=\left\{(H_t)_{t=0}^{T-1}:\ H_t \text{ is a u.m. selector of } \mathcal{H}_t,\ t=0,\dotso, T-1\right\}.$$
Then for any $H\in\mathcal{H}$, $H$ is an adapted process. We make the following assumptions on $\mathcal{H}$.
\begin{assumption}\label{a2} {\ }
\begin{itemize}
\item[(i)] $0\in\mathcal{H}_t(\omega)$, for $\omega\in\Omega_t$, $t=0,\dotso,T-1$.
\item[(ii)] $\mathcal{C}_{\mathcal{H}_t(\omega)}(\mathcal{P}_t(\omega))$ is closed and convex, for $\omega\in\Omega_t$, $t=0,\dotso,T-1$.
\item[(iii)] The set
$$\Psi_{\mathcal{H}_t}:=\{(\omega,Q)\in\Omega_t\times\mathfrak{P}(\Omega):\ E_Q|\Delta S_t(\omega,\cdot)|<\infty\ \text{and }E_Q[y\Delta S_t(\omega,\cdot)]\leq 0,\ \forall y\in\mathcal{H}_t(\omega)\}$$
is analytic, for $t=0,\dotso,T-1$.
\end{itemize}
\end{assumption}

Define
\begin{equation}\label{e5}
\begin{split}
\mathcal{Q}:=&\{Q\in\mathfrak{P}(\Omega_T):\ Q\lll\mathcal{P},\ E_Q[|\Delta S_t|\ |\mathcal{F}_t]<\infty\ Q\text{-a.s. } t=0,\dotso,T-1,\\
&H\cdot S\text{ is a }Q\text{-local-supermartingale}\ \forall H\in\mathcal{H}\}.
\end{split}
\end{equation}
Below is the main theorem of this section:
\begin{theorem}\label{t2}
Under \asref{a2}, NA$(\mathcal{P})$ holds if and only if for each $P\in\mathcal{P}$, there exists $Q\in\mathcal{Q}$ dominating $P$. 
\end{theorem}

\subsection{Proof of \thref{t2}}
We will first provide some auxiliary results.  The following lemma essentially says that if there is no arbitrage in $T$ periods, then there is no arbitrage in any period. It is parallel to \cite[Lemma 4.6]{Nutz2}. Our proof shall mainly focuses on the difference due to the presence of constraints and we put the proof in the appendix. 
\begin{lemma}\label{l6}
Let $t\in\{0,\dotso,T-1\}$. Then the set 
\begin{equation}\label{e17}
N_t:=\{\omega\in\Omega_t:\ \text{NA}(\mathcal{P}_t(\omega))\text{ fails }\}
\end{equation}
is u.m., and if \asref{a2}(i) and NA$(\mathcal{P})$ hold, then $N_t$ is $\mathcal{P}$-polar.
\end{lemma}

The lemma below is a measurable version of \thref{t1}. It is parallel to \cite[Lemma 4.8]{Nutz2}. We provide its proof in the appendix.
\begin{lemma}\label{l15}
Let $t\in\{0,\dotso,T-1\}$, let $P(\cdot):\ \Omega_t\mapsto\mathfrak{P}(\Omega)$ be Borel, and let $\mathcal{Q}_t:\ \Omega_t\Mapsto\mathfrak{P}(\Omega)$,
\begin{equation}\notag
\mathcal{Q}_t(\omega):=\{Q\in\mathfrak{P}(\Omega):\ Q \lll\mathcal{P}_t(\omega),\ E_Q|\Delta S_t(\omega,\cdot)|<\infty,\ E_Q[y\Delta S_t(\omega,\cdot)]\leq 0,\ \forall y\in\mathcal{H}_t(\omega)\}.
\end{equation}
If \asref{a2}(ii)(iii) holds, then $\mathcal{Q}_t$ has an analytic graph and there exist u.m. mappings $Q(\cdot),\hat P(\cdot):\ \Omega_t\rightarrow\mathfrak{P}(\Omega)$ such that 
\begin{eqnarray}
\notag&&P(\omega)\ll Q(\omega)\ll\hat P(\omega)\quad \text{for all }\omega\in\Omega_t,\\
\notag&& \hat P(\omega)\in\mathcal{P}_t(\omega)\quad\text{if }P(\omega)\in\mathcal{P}_t(\omega),\\
\notag&& Q(\omega)\in\mathcal{Q}_t(\omega)\quad \text{if NA}(\mathcal{P}_t(\omega))\text{ holds and }P(\omega)\in\mathcal{P}_t(\omega).
\end{eqnarray}
\end{lemma}

\begin{proof}[\textbf{Proof of \thref{t2}}]
Using Lemmas \ref{l6} and \ref{l15}, we can perform the same glueing argument Bouchard and Nutz use in the proof of \cite[Theorem 4.5]{Nutz2}, and thus we omit it here.
\end{proof}

\subsection{Sufficient conditions for \asref{a2}(iii)}\label{s1}
By \cite[Proposition 7.47]{Shreve}, the map $(\omega,Q)\mapsto\sup_{y\in\mathcal{H}_t(\omega)}E_Q[y\Delta S_t(\omega,\cdot)]$ is u.s.a., which does not necessarily imply the analyticity of $\Psi_{\mathcal{H}_t}$ as the complement of an analytic set may fail to be analytic. Therefore we provide some sufficient conditions for \asref{a2}(iii) below.

\begin{definition}
We call $\mathfrak{H}_t:\ \Omega_t\twoheadrightarrow\mathbb{R}^d$ a stretch of $\mathcal{H}_t$, if for any $\omega\in\Omega_t$, $\mathfrak{C}_{\mathfrak{H}_t(\omega)}=\mathfrak{C}_{\mathcal{H}_t(\omega)}$.
\end{definition}
It is easy to see that for any stretch $\mathfrak{H}_t$ of $\mathcal{H}_t$,
\begin{equation}\notag
\Psi_{\mathcal{H}_t}=\Psi_{\mathfrak{H}_t}=\{(\omega,Q)\in\Omega_t\times\mathfrak{P}(\Omega):\ E_Q|\Delta S_t(\omega,\cdot)|<\infty,\ \sup_{y\in\mathfrak{H}_t(\omega)}yE_Q[\Delta S_t(\omega,\cdot)]\leq 0\}.
\end{equation}
Therefore, in order to show $\Psi_{\cH_t}$ is analytic, it suffices to show that there exists a stretch $\mathfrak{H}_t$ of $\cH_t$, such that the map $\varphi_{\mathfrak{H}_t}:\ \Omega_t\times\mathfrak{P}(\Omega)\mapsto\R^*$
\begin{equation}\label{eq:strngp}
\varphi_{\mathfrak{H}_t}(\omega,Q)=\sup_{y\in\mathfrak{H}_t(\omega)}yE_Q[\Delta S_t(\omega,\cdot)]
\end{equation}
is l.s.a. on $\mathcal{J}:=\{(\omega,Q)\in\Omega_t\times\mathfrak{P}(\Omega):\ E_Q|\Delta S_t(\omega,\cdot)|<\infty\}.$
\begin{proposition}\label{p1}
If there exists a measurable (w.r.t. $\mathcal{B}(\mathbb{R}^d)$) stretch $\mathfrak{H}_t$ of $\mathcal{H}_t$ with nonempty compact values, then $\varphi_{\mathfrak{H}_t}$ is Borel measurable, and thus $\Psi_{\mathcal{H}_t}$ is Borel measurable. 
\end{proposition}
\begin{proof}
The conclusion follows directly from \cite[Theorem 18.19]{IDA}.
\end{proof}
\begin{proposition}\label{p2}
If there exists a stretch $\mathfrak{H}_t$ of $\mathcal{H}_t$ satisfying
\begin{itemize}
\item[(i)] graph$(\mathfrak{H}_t)$ is Borel measurable,
\item [(ii)] there exists a countable set $(y_n)_n\subset\mathbb{R}^d$, such that for any $\omega\in\Omega_t$ and $y\in\mathfrak{H}_t(\omega)$, there exist $(y_{n_k})_k\subset(y_n)_n\cap\mathfrak{H}_t$ converging to $y$,
\end{itemize}
then $\varphi_{\mathfrak{H}_t}$ is Borel measurable, and thus $\Psi_{\mathcal{H}_t}$ is Borel measurable.
\end{proposition}
\begin{proof}
Define function $\phi:\ \R^d\times\mathcal{J}\mapsto\mathbb{R}^*$,
$$
\phi(y,\omega,Q) = \left\{ \begin{array}{rl}
 yE_Q[\Delta S_t(\omega,\cdot)] &\mbox{ if $y\in\mathfrak{H}_t(\omega)$,} \\
  -\infty\ \ \ \ \ \ \ \ \ &\mbox{ otherwise.}
       \end{array} \right.
$$
It can be shown by a monotone class argument that $\phi$ is Borel measurable. So the function $\varphi:\ \mathcal{J}\mapsto\mathbb{R}$
$$\varphi(\omega,Q)=\sup_n\phi(y_n,\omega,Q)$$
is Borel measurable. It remains to show that $\varphi=\varphi_{\mathfrak{H}_t}$.
It is easy to see that $\varphi\geq\varphi_{\mathfrak{H}_t}$. Conversely, take $(\omega,Q)\in\mathcal{J}$. Then $\phi(y_n,\omega,Q)=y_nE_Q[\Delta S(\omega,\cdot)]\leq\varphi_{\mathfrak{H}_t}(\omega,Q)$ if $y_n\in\mathfrak{H}_t(\omega)$, and $\phi(y_n,\omega,Q)=-\infty<\varphi_{\mathfrak{H}_t}(\omega,Q)$ if $y_n\notin\mathfrak{H}_t(\omega)$; i.e., $\varphi(\omega,Q)=\sup_n\phi(y_n,\omega,Q)\leq\varphi_{\mathfrak{H}_t}(\omega,Q)$.
\end{proof}

\begin{example}\label{ex:straex}
Let $\underline a_t^i,\ \overline a_t^i:\ \Omega_t\mapsto\R$ be Borel measurable, with $\underline a_t^i<\overline a_t^i,\ i=1,\dotso,d$. Let
$$\cH_t(\omega)=\prod_{i=1}^d[\underline a_t^i(\omega),\overline a_t^i(\omega)],\quad\omega\in\Omega_t.$$
Then both Propositions~\ref{p1} and \ref{p2} hold with $\mathfrak{H}_t=\cH_t$ and $(y_n)_n=\mathbb{Q}^d$.
\end{example}

\begin{example}\label{ex3.2}
Let $d=1$ and $\cH_t$ be such that for any $\omega\in\Omega_t$, $\cH_t(\omega)\subset (0,\infty)$. We assume that graph$(\cH_t)$ is analytic, but not Borel. Then $\cH_t$ itself does not satisfy the assumptions in Proposition \ref{p1} or \ref{p2}. Now let $\mathfrak{H}_t(\omega)=[1,2],\ \omega\in\Omega_t$. Then $\mathfrak{H}_t$ is a stretch of $\cH_t$, and $\mathfrak{H}_t$ satisfies the assumptions in Propositions \ref{p1} and \ref{p2} with $(y_n)_n=\mathbb{Q}$. 
\end{example}

\section{Super-hedging in one period}
\subsection{The set-up and the main result}
We use the set-up in Section 2. Let $f$ be a u.m. function. Define the super-hedging price
$$\pi^{\mathcal{P}}(f):=\inf\{x:\ \exists H\in\mathcal{H}, \text{ s.t. }x+H\cdot S\geq f,\ \mathcal{P}-q.s.\}.$$
We also denote $\pi^P(f)=\pi^{\{P\}}(f)$. We further assume:
\begin{assumption}\label{a3}
$\mathcal{H}(\mathcal{P})$ is convex and closed.
\end{assumption}
\begin{remark}
It is easy to see that if $\mathcal{H}(\mathcal{P})$ is convex, then $\mathcal{C}_\mathcal{H}(\mathcal{P})$ is convex. 
\end{remark}
Define
$$\mathfrak{Q}:=\{Q\in\mathfrak{P}(\Omega):\ Q\lll\mathcal{P},\ E_Q|\Delta S|<\infty,\ A^Q:=\sup_{H\in\mathcal{H}}E_Q[H\Delta S]<\infty\}.$$
Below is the main result of this section.
\begin{theorem}\label{t3}
Let Assumptions \ref{a1}(ii) \& \ref{a3} and NA$(\mathcal{P})$ hold. Then
\begin{equation}\label{e2}
\pi^{\mathcal{P}}(f)=\sup_{Q\in\mathfrak{Q}}(E_Q[f]-A^Q).
\end{equation}
Besides, $\pi^\cP(f)>-\infty$ and there exists $H\in\mathcal{H}$ such that $\pi^\cP(f)+H\Delta S\geq f\ \mathcal{P}-q.s.$.
\end{theorem}

\subsection{Proof of \thref{t3}}
We first provide two lemmas. 
\begin{lemma}\label{l2}
Let NA$(\mathcal{P})$ hold. If $\mathcal{H}(\mathcal{P})$ and $\mathcal{C}_\mathcal{H}(\mathcal{P})$ are closed, then
$$\pi^{\mathcal{P}}(f)=\sup_{P\in\mathcal{P}}\pi^P(f).$$
\end{lemma}
\begin{proof}
It is easy to see that $\pi^{\mathcal{P}}(f)\geq\sup_{P\in\mathcal{P}}\pi^P(f)$. We shall prove the reverse inequality. If $\pi^{\mathcal{P}}(f)>\sup_{P\in\mathcal{P}}\pi^P(f)$, then there exists $\eps>0$ such that
\begin{equation}\label{e29}
\alpha:=\pi^{\mathcal{P}}(f)\wedge\frac{1}{\eps}-\eps>\sup_{P\in\mathcal{P}}\pi^P(f).
\end{equation}
By \leref{l12} there exists $P''\in\mathcal{P}$, such that $N^\perp(P'')=N^\perp(\mathcal{P})$ and NA$(P'')$ holds. 

Moreover, we have that the set 
$$A_\alpha:=\{H\in\mathcal{H}(\mathcal{P}):\ \alpha+H\Delta S\geq f,\ P''-\text{a.s.}\}$$
is compact. In order to prove this claim take $(H_n)_n\subset A_\alpha$. If $(H_n)_n$ is not bounded, w.l.o.g. we assume $0<||H_n||\rightarrow\infty$; then
\begin{equation}\label{e18}
\frac{\alpha}{||H_n||}+\frac{H_n}{||H_n||}\Delta S\geq \frac{f}{||H_n||}.
\end{equation}
Since $\mathcal{C}_\mathcal{H}(\mathcal{P})$ is closed, there exist some $H\in\mathcal{C}_\mathcal{H}(\mathcal{P})= \mathcal{C}_\mathcal{H}(P'')$ with $||H||=1$ such that $H_{n_k}/||H_{n_k}||\rightarrow H$. Taking the limit along $(n_k)_k$, we have $H\Delta S\geq 0\ P''$-a.s. NA$(P'')$ implies $H\Delta S=0\ P''$-a.s. So $H\in\mathcal{C}_\mathcal{H}(P'')\cap N(P'')=\{0\}$, which contradicts $||H||=1$. Thus $(H_n)_n$ is bounded, and there exists $H''\in\mathbb{R}^d$, such that $(H_{n_j})_j\rightarrow H''$. Since $\mathcal{H}(\mathcal{P})$ is closed, $H''\in\mathcal{H}(\mathcal{P})$, which further implies $H''\in A_\alpha$. 

For any $H\in A_\alpha$, since $\alpha<\pi^\cP(f)$ by \eqref{e29}, there exist $P_H\in\mathcal{P}$ such that
$$P_H(\alpha+H\Delta S<f)>0.$$
It can be further shown that there exists $\delta_H>0$, such that for any $H'\in B(H,\delta_H)$,
$$P_H(\alpha+H'\Delta S<f)>0.$$
Since $A_\alpha\subset\cup_{H\in A_\alpha}B(H,\delta_H)$ and $A_\alpha$ is compact, there exists $(H_i)_{i=1}^n\subset A_\alpha$, such that $A_\alpha\subset\cup_{i=1}^n B(H_i,\delta_{H_i})$. Let
$$P':=\sum_{i=1}^n a_iP_{H_i}+a_0P''\in\mathcal{P},$$
where $\sum_{i=0}^n a_i=1$ and $a_i>0,\ i=0,\dotso,n$.
Then it is easy to see that for any $H\in\mathcal{H}(\mathcal{P})=\mathcal{H}(P'')=\mathcal{H}(P')$,
$$P'(\alpha+H\Delta S<f)>0,$$
which implies that
$$\alpha\leq\pi^{P'}(f)\leq\sup_{P\in\mathcal{P}}\pi^P(f),$$
which contradicts \eqref{e29}.
\end{proof}

\begin{lemma}\label{l3}
Let NA$(\mathcal{P})$ hold. If $\mathcal{H}(\mathcal{P})$ and $\mathcal{C}_\mathcal{H}(\mathcal{P})$ are closed, then the set
\begin{equation}\label{e30}
K(\mathcal{P}):=\{H\Delta S-X:\ H\in\mathcal{H},\ X\in L_+^0(\mathcal{P})\}
\end{equation}
is $\cP-q.s.$ closed. 
\end{lemma}
\begin{proof}
Let $W^n=H^n\Delta S-X^n\in K(\cP)\rightarrow W\ \cP-q.s.$, where w.l.o.g. $H^n\in\mathcal{H}(\cP)$ and $X^n\in L_+^0(\cP),\ n=1,2,\dotso$ If $(H^n)_n$ is not bounded, then without loss of generality, $0<||H^n||\rightarrow\infty$. Consider

\begin{equation}\label{e19}
\frac{W^n}{||H^n||}=\frac{H^n}{||H^n||}\Delta S-\frac{X^n}{||H^n||}.
\end{equation}
As $(H^n/||H^n||)_n$ is bounded, there exists some subsequence $(H^{n_k}/||H^{n_k}||)_k$ converging to some $H\in\R^d$ with $||H||=1$. Taking the limit in \eqref{e19} along $(n_k)_k$, we get that $H\Delta S\geq 0\ \cP-q.s.$. Because $(H^{n_k}/||H^{n_k}||)_k\in\mathcal{C}_\mathcal{H}(\cP)$ and $\mathcal{C}_\mathcal{H}(\cP)$ is closed, $H\in\mathcal{C}_\mathcal{H}(\cP)$. Hence $H\Delta S=0\ \cP-q.s.$ by NA$(\cP)$. Then $H\in\mathcal{C}_\mathcal{H}(\cP)\cap N(\cP)=\{0\}$, which contradicts $||H||=1$.

Therefore, $(H^n)_n$ is bounded and there exists some subsequence $(H^{n_j})_j$ converging to some $H'\in\R^d$. Since $\mathcal{H}(\cP)$ is closed, $H'\in\mathcal{H}(\cP)$. Let $X:=H'\Delta S-W\in L_+^0(\cP)$, then $W=H'\Delta S-X\in K(\cP)$.
\end{proof}
\begin{proof}[\textbf{Proof of \thref{t3}}]
We first show that $\pi^{\mathcal{P}}(f)>-\infty$ and the optimal super-hedging strategy exists. If $\pi^\cP(f)=\infty$ then we are done. If $\pi^\cP(f)=-\infty$, then for any $n\in\mathbb{N}$, there exists $H^n\in\mathcal{H}$ such that
$$H^n\Delta S\geq f+n\geq (f+n)\wedge 1,\quad\cP-q.s.$$
By \leref{l3}, there exists some $H\in\mathcal{H}$ such that $H\Delta S\geq 1\ \cP-q.s.$, which contradicts NA$(\cP)$. If $\pi^\cP(f)\in(-\infty,\infty)$, then for any $n\in\mathbb{N}$, there exists some $\tilde H^n\in\mathcal{H}$, such that $\pi^\cP(f)+1/n+\tilde H^n\Delta S\geq f$. \leref{l3} implies that there exists some $\tilde H\in\mathcal{H}$, such that $\pi^\cP(f)+\tilde H\Delta S\geq f$.

By \leref{l2},
\begin{equation}\label{e3}
\pi^\mathcal{P}(f)=\sup_{P\in\mathcal{P}}\pi^P(f)=\sup_{Q\in\mathcal{Q}}\pi^Q(f)=\sup_{Q\in\mathcal{Q}}\sup_{\substack{Q'\in\mathfrak{Q},\\Q'\sim Q}}(E_{Q'}[f]-A^{Q'})\leq\sup_{Q\in\mathfrak{Q}}(E_Q[f]-A^Q]),
\end{equation}
where we apply \thref{t1} for the second equality, and \cite[Proposition 9.23]{SF} for the third equality. Conversely, if $\pi^\mathcal{P}(f)=\infty$, then we are done. Otherwise let $x>\pi^\mathcal{P}(f)$, and there exist $H\in\mathcal{H}$, such that $x+H\Delta S\geq f\ \mathcal{P}-q.s.$. Then for any $Q\in\mathfrak{Q}$,
$$x\geq E_Q[f]-E_Q[H\Delta S]\geq E_Q[f]-A^Q.$$
By the arbitrariness of $x$ and $Q$, we have that 
$$\pi^\mathcal{P}(f)\geq\sup_{Q\in\mathfrak{Q}}(E_Q[f]-A^Q),$$
which together with \eqref{e3} implies \eqref{e2}.
\end{proof}

\section{ Optional decomposition in multiple periods}
\subsection{The set-up and the main result}
We use the set-up in Section 3. In addition, let $f:\ \Omega_T\mapsto\mathbb{R}$ be u.s.a. We further assume:
\begin{assumption}\label{a4}{\ }
\begin{itemize}
\item[(i)] For $t\in\{0,\dotso,T-1\}$ and $\omega\in\Omega_t$, $(\mathcal{H}_t(\omega))(\mathcal{P}_t(\omega))$ is convex and closed;
\item[(ii)] the map $A_t(\omega,Q):\ \Omega_t\times\mathfrak{P}(\Omega)\mapsto\mathbb{R}^*$,
$$A_t(\omega,Q)=\sup_{y\in\mathcal{H}_t(\omega)}yE_Q[\Delta S_t(\omega,\cdot)]$$
is l.s.a. on the set $\{(\omega,Q):\ E_Q|\Delta S_t(\omega,\cdot)|<\infty\}$.
\end{itemize}
\end{assumption}
\begin{remark}
Observe that $\Psi_{\mathcal{H}_t}$ defined in Assumption~\ref{a2} satisfies
\begin{equation}\label{e76}
\Psi_{\mathcal{H}_t}=\{(\omega,Q)\in\Omega_t\times\mathfrak{P}(\Omega):\ E_Q|\Delta S_t(\omega,\cdot)|<\infty,\ A_t(\omega,Q)\leq 0\}.
\end{equation}
Therefore, \asref{a4}(ii) implies \asref{a2}(iii).
\end{remark}
\begin{remark}
If Proposition ~\ref{p1} or \ref{p2} hold with $\mathfrak{H}_t=\cH_t$, then since $A_t=\varphi_{\mathfrak{H}_t}$ ($\varphi_{\mathfrak{H}_t}$ is defined in \eqref{eq:strngp}), \asref{a4}(ii) holds. See Example~\ref{ex:straex} for a case when this holds.
\end{remark}

For any $Q\in\mathfrak{P}(\Omega_T)$, there are Borel kernels $Q_t:\ \Omega_t\mapsto\mathfrak{P}(\Omega)$ such that $Q=Q_0\otimes\dotso\otimes Q_{T-1}$. For $E^Q[|\Delta S_t|\ |\mathcal{F}_t]<\infty\ Q$-a.s., define $A_t^Q(\cdot):=A_t(\cdot,Q_t(\cdot))$ for $t=0,\dotso,T-1$, and 
$$B_t^Q:=\sum_{i=0}^{t-1}A_i^Q,\quad t=1,\dotso,T$$
and set $B_0^Q=0$. Let 
$$\mathfrak{Q}:=\{Q\in\mathfrak{P}(\Omega_T):\ Q\lll\mathcal{P},\ E_Q[|\Delta S_t|\ |\mathcal{F}_t]<\infty\ Q\text{-a.s. for all } t, \text{ and }B_T^Q<\infty\ Q\text{-a.s.}\}.$$
Then it is not difficult to see that $\mathcal{Q}\subset\mathfrak{Q}$, where $\mathcal{Q}$ is defined in \eqref{e5}.\footnote{A rigorous argument is as follows. Let $Q=Q_0\otimes\dotso\otimes Q_{T-1}\in\cQ$, where $Q_t$ is a Borel kernels, $0\leq t\leq T-1$. It can be shown by a monotone class argument that the map $(\omega,y, Q')\mapsto yE_{Q'}[\Delta S(\omega,\cdot)]$ is Borel measurable for $(\omega,y,Q')\in\Omega_t\times\R^d\times\mathfrak{P}(\Omega)$. Hence the map $(\omega,y)\mapsto yE_{Q_t(\omega)}[\Delta S(\omega,\cdot)]$ is Borel measurable for $(\omega,y)\in\Omega_t\times\R^d$. Since Graph$(\cH_t)$ is analytic, by \cite[Proposition 7.50]{Shreve} there exists a u.m. selector $H_t^n(\cdot)\in\cH_t(\cdot)$, such that
$$A_t^Q(\omega)\wedge n-1/n\leq H_t^n(\omega)E_{Q_t(\omega)}[\Delta S_t(\omega,\cdot)]\leq 0,\ \text{for}\ Q\text{-a.s.}\ \omega\in\Omega_t,$$
where the second inequality follows from the local-supermartingale property of $H^n\cdot S$ with $H^n=(0,\dotso,0,H_t^n,0\dotso,0)\in\cH$. Sending $n\rightarrow\infty$ we get that $A_t^Q\leq 0\ Q$-a.s. for $t=0,\dotso,T-1$, and thus $Q\in\mathfrak{Q}$.
} Also if for each $t\in\{0,\dotso,T-1\}$ and $\omega\in\Omega_t$, $\mathcal{H}_t(\omega)$ is a convex cone, then $\mathfrak{Q}=\mathcal{Q}$. Below is the main result of this section.
\begin{theorem}\label{t4}
Let Assumptions \ref{a2} \& \ref{a4} and NA$(\mathcal{P})$ hold. Let $V$ be an adapted process such that $V_t$ is u.s.a. for $t=1,\dotso,T$. Then the following are equivalent:
\begin{itemize}
\item [(i)] $V-B^Q$ is a $Q$-local-supermartingale for each $Q\in\mathfrak{Q}$.
\item [(ii)] There exists $H\in\mathcal{H}$ and an adapted increasing process $C$ with $C_0=0$ such that
$$V_t=V_0+(H\cdot S)_t-C_t,\quad\mathcal{P}-q.s.$$
\end{itemize}
\end{theorem}

\subsection{Proof of \thref{t4}}
We first provide three lemmas for the proof of \thref{t4}. We shall prove Lemmas \ref{l9}\ \&\ \ref{l7} in the appendix.
\begin{lemma}\label{l9}
Let \asref{a4}(ii) hold, and define $\mathfrak{Q}_t:\ \Omega_t\twoheadrightarrow\mathfrak{P}(\Omega)$ by
\begin{equation}\label{e6}
\mathfrak{Q}_t(\omega):=\{Q\in\mathfrak{P}(\Omega):\ Q\lll\mathcal{P}_t(\omega),\ E_Q|\Delta S_t(\omega,\cdot)|<\infty,\ A_t(\omega,Q)<\infty\}.
\end{equation}
Then $\mathfrak{Q}_t$ has an analytic graph.
\end{lemma}

The following lemma, which is a measurable version of \thref{t3}, is parallel to \cite[Lemma 4.10]{Nutz2}. Given \thref{t3}, the proof of this lemma follows exactly the argument of \cite[Lemma 4.10]{Nutz2}, and thus we omit it here.
\begin{lemma}\label{l8}
Let NA$(\mathcal{P})$ and \asref{a4} hold, and let $t\in\{0,\dotso,T-1\}$ and $\hat f:\ \Omega_t\times\Omega\mapsto\mathbb{R}^*$ be u.s.a.. Then
$$\mathcal{E}_t(\hat f):\ \Omega_t\mapsto\mathbb{R}^*,\quad\mathcal{E}_t(\hat f)(\omega):=\sup_{Q\in\mathfrak{Q}_t(\omega)}(E_Q[\hat f(\omega,\cdot)]-A_t(\omega,Q))$$
is u.s.a.. Besides, there exists a u.m. function $y(\cdot):\ \Omega_t\mapsto\mathbb{R}^d$ with $y(\cdot)\in\mathcal{H}_t(\cdot)$, such that
$$\mathcal{E}_t(\hat f)(\omega)+y(\omega)\Delta S_t(\omega,\cdot)\geq \hat f(\omega,\cdot)\quad \mathcal{P}_t(\omega)-q.s.$$
for all $\omega\in\Omega_t$ such that NA$(\mathcal{P}_t(\omega))$ holds and $\hat f(\omega,\cdot)>-\infty\ \mathcal{P}_t(\omega)-q.s.$.
\end{lemma}

\begin{lemma}\label{l7}
Let Assumptions \ref{a2} \& \ref{a4} and NA$(\mathcal{P})$ hold. Recall $\mathfrak{Q}_t$ defined in \eqref{e6}. We have that
\begin{equation}\notag
\mathfrak{Q}=\big\{Q_0\otimes\dotso\otimes Q_{T-1}:\ Q_t(\cdot)\text { is a u.m. selector of }\mathfrak{Q}_t,\ t=0,\dotso,T-1\big\}.
\end{equation}
\end{lemma}

\begin{proof}[\textbf{Proof of \thref{t4}}]
(ii)$\implies$(i): For any $Q\in\mathfrak{Q}$,
$$V_{t+1}=V_t+H_t\Delta S_t-(C_{t+1}^Q-C_t^Q)\leq V_t+H_t\Delta S_t,\ Q\text{-a.s.}.$$
Hence,
$$E_Q[V_{t+1}|\mathcal{F}_t]\leq V_t+H_tE_Q[\Delta S_t|\mathcal{F}_t]\leq V_t+A_t^Q=V_t+B_{t+1}^Q-B_t^Q,$$
i.e., 
$$E_Q[V_{t+1}-B_{t+1}^Q|\mathcal{F}_t]\leq V_t-B_t^Q.$$
(i)$\implies$(ii): We shall first show that
\begin{equation}\label{e8}
\mathcal{E}_t(V_{t+1})\leq V_t,\quad\mathcal{P}-q.s.
\end{equation}
Let $Q=Q_1\otimes\dotso\otimes Q_{T-1}\in\mathfrak{Q}$ and $\eps>0$. The map $(\omega,Q)\rightarrow E_Q[V_{t+1}(\omega,\cdot)]-A_t(\omega,Q)$ is u.s.a., and graph$(\mathfrak{Q}_t)$ is analytic. As a result, by \cite[Proposition 7.50]{Shreve} there exists a u.m. selector $Q_t^\eps:\Omega_t\mapsto\mathfrak{P}(\Omega)$, such that $Q_t^\eps(\cdot)\in\mathfrak{Q}_t(\cdot)$ on $\{\mathfrak{Q}_t\neq\emptyset\}$ (whose complement is a $Q$-null set), and
$$E_{Q_t^\eps(\cdot)}[V_{t+1}]-A_t(\cdot,Q_t^\eps(\cdot))\geq \mathcal{E}_t(V_{t+1})\wedge\frac{1}{\eps}-\eps,\quad Q\text{-a.s.}$$
Define $$Q'=Q_1\otimes\dotso\otimes Q_{t-1}\otimes Q_t^\eps\otimes Q_{t+1}\otimes Q_{T-1}.$$
Then $Q'\in\mathfrak{Q}$ by \leref{l7}. Therefore,
$$E_{Q'}[V_{t+1}-B_{t+1}^{Q'}|\mathcal{F}_t]\leq V_t-B_t^{Q'},\quad Q'\text{-a.s.}$$
Noticing that $Q=Q'$ on $\Omega_t$, we have
$$V_t\geq E_{Q'}[V_{t+1}|\mathcal{F}_t]-A_t^{Q'}=E_{Q_t^\eps(\cdot)}[V_{t+1}]-A_t(\cdot,Q_t^\eps(\cdot))\geq \mathcal{E}_t(V_{t+1})\wedge\frac{1}{\eps}-\eps,\quad Q\text{-a.s.}.$$
By the arbitrariness of $\eps$ and $Q$, we have \eqref{e8} holds.

By \leref{l8}, there exists a u.m. function $H_t:\ \Omega_t\mapsto\mathbb{R}^d$ such that
$$\mathcal{E}_t(V_{t+1})(\omega)+H_t(\omega)\Delta S_{t+1}(\omega,\cdot)\geq V_{t+1}(\omega,\cdot)\quad\mathcal{P}_t(\omega)-q.s.$$
for $\omega\in\Omega_t\setminus N_t$. 
Fubini's theorem and \eqref{e8} imply that
$$V_t+H_t\Delta S_t\geq V_{t+1}\quad\mathcal{P}-q.s..$$
Finally, by defining $C_t:=V_0+(H\cdot S)_t-V_t$, the conclusion follows. 
\end{proof}

\section{Hedging European and American options in multiple periods}

\subsection{Hedging European options}
Let $f:\ \Omega_T\mapsto\mathbb{R}$ be a u.s.a. function, which represents the payoff of a European option. Define the super-hedging price
$$\pi(f):=\inf\{x:\ \exists H\in\mathcal{H},\ \text{s.t. }x+(H\cdot S)_T\geq f,\ \mathcal{P}-q.s.\}.$$
\begin{theorem}\label{t5}
Let Assumptions \ref{a2} \& \ref{a4} and NA$(\mathcal{P})$ hold. Then the super-hedging price is given by
\begin{equation}\label{e9}
\pi(f)=\sup_{Q\in\mathfrak{Q}}\left(E_Q[f]-E_Q[B_T^Q]\right).
\end{equation}
Moreover, $\pi(f)>-\infty$ and there exists $H\in\mathcal{H}$, such that $\pi(f)+(H\cdot S)_T\geq f\ \cP-q.s.$.
\end{theorem}
\begin{proof}
It is easy to see that $\pi(f)\geq \sup_{Q\in\mathfrak{Q}}(E_Q[f]-E_Q[B_T^Q])$. We shall show the reverse inequality. Define $V_T=f$ and 
$$V_{t}=\mathcal{E}_t(V_{t+1}),\ t=0,\dotso,T-1.$$
Then $V_t$ is u.s.a. by \leref{l8} for $t=1,\dotso,T$. It is easy to see that $(V_t-B^Q_t)_t$ is a $Q$-local-supermartingale for each $Q\in\mathfrak{Q}$. Then by \thref{t4}, there exists $H\in\mathcal{H}$, such that
$$V_0+(H\cdot S)_T\geq V_T=f,\quad\mathcal{P}-q.s.$$
Hence $V_0\geq\pi(f)$. It remains to show that
\begin{equation}\label{e11}
V_0\leq\sup_{Q\in\mathfrak{Q}}\left(E_Q[f]-E_Q[B_T^Q]\right).
\end{equation}

First assume that $f$ is bounded from above. Then by \cite[Proposition 7.50]{Shreve}, \leref{l9} and \leref{l8}, we can choose a u.m. $\eps$ optimizer $Q_t^\eps$ for $\mathcal{E}_t$ in each time period. Define $Q^\eps:=Q_0^\eps\otimes\dotso\otimes Q_{T-1}^\eps\in\mathfrak{Q}$,  
$$V_0=\mathcal{E}_0\circ\dotso\circ\mathcal{E}_{T-1}(f)\leq E_{Q^\eps}[f-B_T^{Q^\eps}]+T\eps\leq\sup_{Q\in\mathfrak{Q}}E_Q[f-B_T^Q]+T\eps,$$
which implies \eqref{e11}. 

In general let $f$ be any u.s.a. function. Then we have
$$\mathcal{E}_0\circ\dotso\circ\mathcal{E}_{T-1}(f\wedge n)\leq\sup_{Q\in\mathfrak{Q}}\left(E_Q[f\wedge n]-E_Q[B_T^Q]\right).$$
Obviously the limit of the right hand side above is $\sup_{Q\in\mathfrak{Q}}\left(E_Q[f]-E_Q[B_T^Q]\right)$. To conclude that the limit of the left hand side is $\mathcal{E}_0\circ\dotso\circ\mathcal{E}_{T-1}(f)$, it suffices to show that for any $t\in\{0,\dotso,T-1\}$, and $\mathcal{F}_{t+1}$-measurable functions $v^n\nearrow v$, 
$$\gamma:=\sup_n\mathcal{E}_t(v^n)=\mathcal{E}_t(v),\ \cP-q.s..$$
Indeed, for $\omega\in \Omega_t\setminus N_t$,  by \thref{t3} $v^n(\omega)-\gamma(\omega)\in K(\cP(\omega))$, where $N_t$ and $K(\cdot)$ are defined in \eqref{e17} and \eqref{e30} respectively. Since $K(\cP(\omega))$ is closed by  \leref{l3}, $v(\omega)-\gamma(\omega)\in K(\cP(\omega))$, which implies $\gamma(\omega)\geq\mathcal{E}_t(v)(\omega)$ by \thref{t3}.

Finally, using a backward induction we can show that $V_t>-\infty\ \cP-q.s.,\ t=0,\dotso,T-1$ by \leref{l6} and \thref{t3}. In particular, $\pi(f)=V_0>-\infty$.
\end{proof}

\begin{corollary}\label{c1}
Let \asref{a4} and NA$(\mathcal{P})$ hold. Assume that for any $t\in\{0,\dotso,T-1\}$ and $\omega\in\Omega_t$, $\mathcal{H}_t(\omega)$ is a convex cone containing the origin. Then
$$\pi(f)=\sup_{Q\in\mathcal{Q}}E_Q[f].$$
\end{corollary}
\begin{proof}
This follows from \eqref{e76} and that $\mathfrak{Q}=\mathcal{Q}$ and $B_T^Q=0$ for any $Q\in\mathcal{Q}$.
\end{proof}

\subsection{Hedging American options}
We consider the sub- and super-hedging prices of an American option in this subsection. The same problems are analyzed in \cite{BHZ} but without portfolio constraints. The analysis here is essentially the same, so we only provide the results and the main ideas for their proofs. For more details and discussion see \cite{BHZ}.

For $t\in\{0,\dotso,T-1\}$ and $\omega\in\Omega_t$, define
$$\mathfrak{Q}^t(\omega):=\{Q_t(\omega)\otimes\dotso\otimes Q_{T-1}(\omega,\cdot):\ Q_i\text { is a u.m. selector of }\mathfrak{Q}_i,\ i=t,\dotso,T-1\}.$$
In particular $\mathfrak{Q}^0=\mathfrak{Q}$. Assume graph$(\mathfrak{Q}^t)$ is analytic. Let $\mathcal{T}$ be the set of stopping times with respect to the raw filtration $(\mathcal{B}(\Omega_t))_t$, and let $\mathcal{T}_t\subset\mathcal{T}$ be the set of stopping times that are no less than $t$. 

Let $\f=(f_t)_t$ be the payoff of the American option. Assume that $\f_t\in\mathcal{B}(\Omega_t),\ t=1,\dotso,T$, and $\f_\tau\in L^1(Q)$ for any $\tau\in\mathcal{T}$ and $Q\in\fQ$. Define the sub-hedging price:
$$\underline\pi(\f):=\sup\{x:\ \exists (H,\tau)\in\mathcal{H}\times\mathcal{T},\ \text{s.t.}\ \f_\tau+(H\cdot S)_\tau\geq x,\ \mathcal{P}-q.s.\},$$
and the super-hedging price:
$$\overline\pi(\f):=\inf\{x:\ \exists H\in\mathcal{H},\text{ s.t. } x+(H\cdot S)_\tau\geq \f_\tau,\ \mathcal{P}-q.s.,\ \forall \tau\in\mathcal{T}\}.$$

\begin{proposition}
(i) The sub-hedging price is given by
\begin{equation}\label{e20}
\underline\pi(\f)=\sup_{\tau\in\T}\inf_{Q\in\fQ}E_Q[\f_\tau+B_T^Q].
\end{equation}
(ii) For $t\in\{1,\dotso,T-1\}$, assume that the map 
$$\phi_t:\ \Omega_t\times\mathfrak{P}(\Omega_{T-t})\mapsto\mathbb{R}^*,\ \phi_t(\omega,Q)=\sup_{\tau\in\T_t}E_Q\left[\f_\tau(\omega,\cdot)-\sum_{i=t}^{\tau-1}A_i^Q(\omega,\cdot)\right]$$
is u.s.a. Then 
\begin{equation}\label{e21}
\overline\pi(\f)=\sup_{\tau\in\T}\sup_{Q\in\fQ}E_Q[\f_\tau-B_\tau^Q],
\end{equation}
and there exists $H\in\mathcal{H}$, such that $\overline\pi(f)+(H\cdot S)_\tau\geq\f_\tau, \cP-q.s.,\ \forall\tau\in\mathcal{T}$.
\end{proposition}
\begin{proof}
(i) We first show that
$$\underline\pi(f)=\sup\{x:\ \exists (H,\tau)\in\mathcal{H}\times\mathcal{T},\ \text{s.t.}\ \f_\tau+(H\cdot S)_T\geq x,\ \mathcal{P}-q.s.\}=:\beta.$$
For any $x<\underline\pi(f)$, there exists $(H,\tau)\in\mathcal{H}\times\mathcal{T}$, such that $\f_\tau+(H\cdot S)_\tau\geq x\ \cP-q.s.$. Define $H':=(H_t1_{\{t<\tau\}})_t$. For $t=0,\dotso,T-1$, since $\{t<\tau\}\in\mathcal{B}(\Omega_t)$, $H'_t(\cdot)$ is u.m.; besides, $H'_t(\cdot)$ is equal to either $H_t(\cdot)\in\cH_t(\cdot)$ or $0\in\cH_t(\cdot)$. Hence $H'\in\cH$. Then $\f_\tau+(H'\cdot S)_T=\f_\tau+(H\cdot S)_\tau\geq x\ \cP-q.s$, which implies $x\leq\beta$, and thus $\underline\pi(f)\leq\beta$. 

Conversely, for $x<\beta$, there exists $(H,\tau)\in\mathcal{H}\times\mathcal{T}$, such that $\f_\tau+(H\cdot S)_T\geq x\ \cP-q.s.$ Then we also have that $\f_\tau+(H\cdot S)_\tau\geq x\ \cP-q.s.$. To see this, let us define $D:=\{\f_\tau+(H\cdot S)_\tau<x\}$ and $H':=(H_t1_{\{t\geq\tau\}\cap D})_t\in\mathcal{H}$. We get that
$$(H'\cdot S)_T=[(H\cdot S)_T-(H\cdot S)_\tau]1_D\geq 0\ \cP-q.s.,\text{ and }(H'\cdot S)_T>0\ \cP-q.s. \text{ on}\ D.$$
NA$(\cP)$ implies $D$ is $\cP$-polar. Therefore $x\leq\underline\pi(f)$, and thus $\beta\leq\underline\pi(f)$.

It can be shown that
$$\underline\pi(\f)=\beta=\sup_{\tau\in\T}\sup\{x:\ \exists H\in\mathcal{H}:\ \f_\tau+(H\cdot S)_T\geq x,\ \mathcal{P}-q.s.\}=\sup_{\tau\in\T}\inf_{Q\in\fQ}E_Q[\f_\tau+B_T^Q],$$
where we apply \thref{t5} for the last equality above.\\
(ii) Define
$$V_t:\ \Omega_t\mapsto\mathbb{R}^*,\ \ \ V_t=\sup_{Q\in\mathfrak{Q}^t}\sup_{\tau\in\mathcal{T}_t}E_Q\left[\f_\tau(\omega,\cdot)-\sum_{i=t}^{\tau-1}A_i^Q(\omega,\cdot)\right].$$
It can be shown that $V_t$ is u.s.a. for $t=1,\dotso,T$ and $(V_t-B_t^Q)_t$ is a $Q$-supermartingale for each $Q\in\mathfrak{Q}$. By \thref{t4}, there exists $H\in\mathcal{H}$ such that
$$V_0+(H\cdot S)_\tau\geq\f_\tau, \mathcal{P}-q.s.,\ \forall\tau\in\T.$$
Therefore, $\sup_{\tau\in\T}\sup_{Q\in\fQ}E_Q[\f_\tau-B_\tau^Q]=V_0\leq\overline\pi(\f)$. The reverse inequality is easy to see.
\end{proof}

\begin{remark}
In \eqref{e20} and \eqref{e21}, the penalization terms are $B_T^Q$ and $B_\tau^Q$ respectively. In fact, similar to the argument in (i) above, one can show that
\begin{eqnarray}
\hat\pi(f)&:=&\inf\{x:\ \forall\tau\in\T,\ \exists H\in\mathcal{H},\text{ s.t. } x+(H\cdot S)_\tau\geq \f_\tau,\ \mathcal{P}-q.s.\}\notag\\
&=&\sup_{\tau\in\T}\inf\{x:\ \exists H\in\mathcal{H},\text{ s.t. } x+(H\cdot S)_\tau\geq \f_\tau,\ \mathcal{P}-q.s.\}\notag\\
&=&\sup_{\tau\in\T}\inf\{x:\ \exists H\in\mathcal{H},\text{ s.t. } x+(H\cdot S)_T\geq \f_\tau,\ \mathcal{P}-q.s.\}\\
&=&\sup_{\tau\in\T}\sup_{Q\in\fQ}E_Q[\f_\tau-B_T^Q]\notag
\end{eqnarray}
Even though the definition of $\hat\pi(f)$ is less useful for super-hedging since the stopping time should not be known in advance, it suggests that $B_T^Q$ comes from knowing $\tau$ in advance (compare $\underline\pi(f)$ and $\hat\pi(f)$). It is also both mathematically and financially meaningful that $\hat\pi(f)\leq\overline\pi(f)$. However, it is interesting that when $B^Q$ vanishes (e.g., when $\mathcal{H}_t(\cdot)$ is a cone), then $\hat\pi(f)=\overline\pi(f)$.
\end{remark}

\section{FTAP and super-hedging in multiple periods with options}

Let us use the set-up in Section 3. In addition, let $g=(g^1,\dotso,g^e):\ \Omega_T\mapsto\mathbb{R}^e$ be Borel measurable, and each $g^i$ is seen as an option which can and only can be traded at time $t=0$ without constraints. Without loss of generality we assume the price of each option is $0$. In this section, we say NA$(\cP)^g$ holds if for any $(H,h)\in\cH\times\mathbb{R}^e$,
$$(H\cdot S)_T+hg\geq 0\ \cP-q.s.\ \Longrightarrow\ (H\cdot S)_T+hg=0\ \cP-q.s..$$
Obviously NA$(\cP)^g$ implies NA$(\cP)$. 

\begin{definition}
$f:\ \Omega_T\mapsto\mathbb{R}$ is replicable (by stocks and options), if there exists some $x\in\mathbb{R},\ h\in\mathbb{R}^e$ and $H\in\mathcal{H}$, such that
$$x+(H\cdot S)_T+hg=f\quad\text{or}\quad x+(H\cdot S)_T+hg=-f.$$
\end{definition}
Let
$$\mathcal{Q}_g:=\{Q\in\mathcal{Q}:\ E_Q[g]=0\}.$$
Below is the main result of this section:
\begin{theorem} Let assumptions in \coref{c1} hold. Also assume that $g^i$ is not replicable by stocks and other options, and $g^i\in L^1(\cQ),\ i=1,\dotso,e$. Then we have the following.\\
(i) NA$(\mathcal{P})^g$ holds if and only if for each $P\in\mathcal{P}$, there exists $Q\in\mathcal{Q}_g$ dominating $P$.\\
(ii) Let NA$(\mathcal{P})^g$ holds. Let $f:\ \Omega_T\mapsto\mathbb{R}$ be Borel measurable such that $f\in L^1(\cQ)$. Then
\begin{equation}\label{e22}
\pi(f):=\inf\{x\in\mathbb{R}:\ \exists(H,h)\in\mathcal{H}\times\mathbb{R}^e\text{ s.t. }x+(H\cdot S)_T+hg\geq f,\ \mathcal{P}-q.s.\}=\sup_{Q\in\mathcal{Q}_g}E_Q[f].
\end{equation}
Moreover, there exists $(H,h)\in\mathcal{H}\times\R^e$, such that $\pi(f)+(H\cdot S)_T+hg\geq f\ \mathcal{P}-q.s.$.\\
(iii) Assume in addition $\mathcal{H}=-\mathcal{H}$. Let NA$(\mathcal{P})^g$ hold and $f:\ \Omega_T\mapsto\mathbb{R}$ be Borel measurable satisfying $f\in L^1(\cQ_g)$. Then the following are equivalent:
\begin{itemize}
\item[(a)] $f$ is replicable;
\item[(b)] The mapping $Q\mapsto E_Q[f]$ is a constant on $\mathcal{Q}_g$;
\item[(c)] For all $P\in\mathcal{P}$ there exists $Q\in\mathcal{Q}_g$ such that $P\ll Q$ and $E_Q[f]=\pi(f)$.
\end{itemize}
Moreover, the market is complete\footnote{That is, for any Borel measurable function $f:\ \Omega_T\mapsto\mathbb{R}$ satisfying $f\in L_g^1(\cQ)$, $f$ is replicable.}if and only if $\mathcal{Q}_g$ is a singleton.
\end{theorem}
\begin{proof}
We first show the existence of an optimal super-hedging strategy in (ii). It can be shown that
$$\pi(f)=\inf_{h\in\R^e}\inf\{x\in\mathbb{R}:\ \exists H\in\mathcal{H}\text{ s.t. }x+(H\cdot S)_T\geq f-hg,\ \mathcal{P}-q.s.\}=\inf_{h\in\R^e}\sup_{Q\in\mathcal{Q}}E_Q[f-hg],$$
where we apply \coref{t5} for the second equality above. 

We claim that $0$ is a relative interior point of the convex set 
$$\mathcal{I}:=\{E_Q[g]:\ Q\in\cQ\}.$$
If not, then there exists some $h\in\R^e$ with $h\neq 0$, such that $E_Q[hg]\leq 0$ for any $Q\in\cQ$. Then the super-hedging price of $hg$ using $S$, $\pi^0(hg)$, satisfies $\pi^0(hg)\leq 0$ by \coref{c1}. Hence by \thref{t5} there exists $H\in\mathcal{H}$, such that $(H\cdot S)_T\geq hg\ \cP-q.s.$. As the price of $hg$ is $0$, NA$(\cP)^g$ implies that
$$(H\cdot S)_T-hg=0\ \cP-q.s.,$$
which contradicts the assumption that each $g^i$ cannot be replicated by $S$ and the other options, as $h\neq 0$. Hence we have shown that $0$ is a relative interior point of $\mathcal{I}$.

Define $\phi:\ \mathbb{R}^e\mapsto\R$,
$$\phi(h)=\sup_{Q\in\mathcal{Q}}E_Q[f-hg],$$ and
observe that 
$$\pi(f)=\inf_{h\in\R^e}\phi(h)=\inf_{h\in\text{span}(\mathcal{I})}\phi(h).$$
We will now show that there exists a compact set $\mathbb{K}\subset \text{span}(\mathcal{I})$, such that
\begin{equation}\label{e27}
\pi(f)=\inf_{h\in\mathbb{K}}\phi(h).
\end{equation}
In order to do this, we will show that for any $h$ outside a particular ball will satisfy $\phi(h) \geq \phi(0)$, which establishes the claim.

Now, since $0$ is a relative interior point of $\mathcal{I}$, there exists $\gamma>0$, such that 
$$B_{\gamma}:=\{v\in\text{span}(\mathcal{I}):\ ||v||\leq\gamma\}\subset\mathcal{I}.$$
Consider the ball $\mathbb{K}:=\{h\in\text{span}(\mathcal{I}):\ ||h||\leq 2\sup_{Q\in\cQ}E_Q|f|/\gamma\}$.
Then for any $h\in\text{span}(\mathcal{I}) \setminus\mathbb{K}$, there exists $Q\in\cQ$ such that $-hE_Q[g]>2\sup_{Q\in\cQ}E_Q|f|$ (pick $Q$ s.t. $E_Q[g]$ is in the same direction as $-h$ and lies on the circumference of $B_{\gamma}$). This implies that 
$$\phi(h)\geq\sup_{Q\in\cQ}E_Q[-hg]-\sup_{Q\in\cQ}E_Q|f|>\sup_{Q\in\cQ}E_Q|f|=\phi(0).$$
Since such $h$ are suboptimal, it follows that
$$\pi(f)=\inf_{h\in\mathbb{K}}\phi(h).$$

On the other hand, observe that
$$|\phi(h)-\phi(h')|\leq \sup_{Q\in\mathcal{Q}}|E_Q[f-hg]-E_Q[f-h'g]|\leq\sup_{Q\in\mathcal{Q}}E|(h-h')g|\leq ||h-h'||\sup_{Q\in\mathcal{Q}}E_Q[||g||],$$
i.e. $\phi$ is continuous (in fact Lipschitz). Hence there exists some $h^*\in\mathbb{K}\subset\R^e$, such that
$$\pi(f)=\inf_{h\in\R^e}\sup_{Q\in\mathcal{Q}}E_Q[f-hg]=\sup_{Q\in\mathcal{Q}}E_Q[f-h^*g]=\inf\{x\in\mathbb{R}:\ \exists H\in\mathcal{H}\text{ s.t. }x+H\cdot S\geq f-h^*g,\ \mathcal{P}-q.s.\}.$$
Then by \thref{t5} there exists $H^*\in\mathcal{H}$, such that $\pi(f)+(H^*\cdot S)_T\geq f-h^*g\ \cP-q.s.$.

Next let us prove (i) and \eqref{e22} in (ii) simultaneously by induction. For $e=0$, (i) and \eqref{e22} hold by \thref{t1} and \coref{c1}. Assume for $e=k$ (i) and \eqref{e22} hold and we consider $e=k+1$. We first consider (i).  Let $\pi^k(g^{k+1})$ be the super-hedging price of $g^{k+1}$ using stocks $S$ and options $g':=(g^1,\dotso, g^k)$. By induction hypothesis, we have 
$$\pi^k(g^{k+1})=\sup_{Q\in\cQ_{g'}}E_Q[g^{k+1}].$$
Recall that the price of $g^{k+1}$ is $0$. Then NA$(\cP)^g$ implies $\pi^k(g^{k+1})\geq 0$. If $\pi^k(g^{k+1})=0$, then there exists $(H,h)\in\mathcal{H}\times\R^k$, such that $(H\cdot S)_T+hg'-g^{k+1}\geq 0\ \cP-q.s.$. Then by NA$(\cP)^g$,
$$(H\cdot S)_T+hg'-g^{k+1}=0,\quad\cP-q.s.,$$
which contradicts the assumption that $g^{k+1}$ cannot be replicated by $S$ and $g'$. Therefore, $\pi^k(g^{k+1})>0$. Similarly $\pi^k(-g^{k+1})>0$. Thus we have
$$\inf_{Q\in\cQ_{g'}}E_Q[g^{k+1}]<0<\sup_{Q\in\cQ_{g'}}E_Q[g^{k+1}].$$
Then there exists $Q_-,Q_+\in\cQ_{g'}$ satisfying
\begin{equation}\label{e24}
E_{Q_-}[g^{k+1}]<0<E_{Q_+}[g^{k+1}].
\end{equation}
Then for any $P\in\cP$, let $Q\in\cQ_{g'}$ dominating $P$. Let
$$Q':=\lambda_-Q_-+\lambda Q+\lambda_+Q_+.$$
By choosing some appropriate $\lambda_-,\lambda,\lambda_+>0$ with $\lambda_-+\lambda+\lambda_+=1$, we have $P\ll Q'\in\cQ_g$, where $g=(g^1,\dotso,g^{k+1})$.

Next consider \eqref{e22} in (ii). Denote the super-hedging price $\pi^k(\cdot)$ when using $S$ and $g'$, and $\pi(\cdot)$ when using $S$ and $g$, which is consistent with the definition in \eqref{e22}. It is easy to see that 
\begin{equation}\label{e26}
\pi(f)\geq\sup_{Q\in\mathcal{Q}_g}E_Q[f],
\end{equation}
and we focus on the reverse inequality. It suffices to show that
\begin{equation}\label{e23}
\exists Q_n\in\cQ_{g'}, \text{ s.t. }E_{Q_n}[g^{k+1}]\rightarrow 0\text{ and }E_{Q_n}[f]\rightarrow \pi(f).
\end{equation}
Indeed, if \eqref{e23} holds, then we define
$$Q_n':=\lambda_-^nQ_-+\lambda^n Q_n+\lambda_+^nQ_+,\quad\text{s.t. }E_{Q_n'}[g^{k+1}]=0,\text{ i.e., }Q_n'\in\cQ_g,$$
where $Q_+,Q_-$ are from \eqref{e24} and $\lambda_-^n,\lambda^n,\lambda_+^n\in[0,1]$ such that $\lambda_-^n+\lambda^n+\lambda_+^n=1$. Since $E_{Q_n}[g^{k+1}]\rightarrow 0$, we can choose $\lambda^n_{\pm}\rightarrow 0$. Then $E_{Q_n'}[f]\rightarrow\pi(f)$, which implies $\pi(f)\leq\sup_{Q\in\mathcal{Q}_g}E_Q[f]$. 

So let us concentrate on proving \eqref{e23}. By a translation, we may w.l.o.g. assume $\pi(f)=0$. Thus if \eqref{e23} fails, we have
$$0\notin\overline{\{E_Q[(g^{k+1},f)]:\ Q\in\cQ_{g'}\}}\subset\R^2.$$
Then there exists a separating vector $(y,z)\in\R^2$ with $||(y,z)||=1$ such that
\begin{equation}\label{e25}
\sup_{Q\in\cQ_{g'}}E_Q[yg^{k+1}+zf]<0.
\end{equation}
By the induction hypothesis, we have that
$$0>\sup_{Q\in\cQ_{g'}}E_Q[yg^{k+1}+zf]=\pi^k(yg^{k+1}+zf)\geq\pi(yg^{k+1}+zf)=\pi(zf).$$
Obviously from the above $z\neq 0$. If $z>0$, then by positive homogeneity $\pi(f)<0$, contradicting the assumption $\pi(f)=0$. Hence $z<0$. Take $Q''\in\cQ_g\subset\cQ_{g'}$. Then by \eqref{e25} $0>E_{Q''}[yg^{k+1}+zf]=E_{Q''}[zf]$, and thus $E_{Q''}[f]>0=\pi(f)$, which contradicts \eqref{e26}.

Finally, let us prove (iii). It is easy to see that (a)$\implies$(b)$\implies$(c). Now let (c) hold. Let $(H,h)\in\mathcal{H}\times\R^e$ such that $\pi(f)+(H\cdot S)_T+hg\geq f\ \mathcal{P}-q.s.$ If there exists $P\in\mathcal{P}$ satisfying
$$P\left\{\pi(f)+(H\cdot S)_T+hg>f\right\}>0,$$
then by choosing a $Q\in\mathcal{Q}_g$ that dominates $P$, we would have that $\pi(f)>E_Q[f]=\pi(f)$, contradiction. Hence $\pi(f)+H\cdot S+hg=f\ \mathcal{P}-q.s.$, i.e., $f$ is replicable. 

If the market is complete, then by letting $f=1_A$, we know that $Q\mapsto Q(A)$ is constant on $\mathcal{Q}$ for every $A\in\mathcal{B}(\Omega)$ by (b). As any probability measure is uniquely determined by its value on $\mathcal{B}(\Omega)$, we know that $\mathcal{Q}$ is a singleton. Conversely, if $\mathcal{Q}$ is a singleton, then (b) holds, and thus the market is complete by (a).
\end{proof}

\begin{appendices}
\section{Proofs of Some Technical Results}
\subsection{Proof of \leref{l6}}
\begin{proof}
Fix $t\in\{0,\dotso,T-1\}$ and let
\begin{equation}\label{e70}
\Lambda^\circ(\omega):=\{y\in\mathbb{R}^d:\ yv\geq0,\ \text{for all}\ v\in\text{supp}_{\mathcal{P}(\omega)}(\Delta S_t(\omega,\cdot))\},\quad\omega\in\Omega_t.
\end{equation}
It could be easily shown that
$$N_t^c=\{\omega\in\Omega_t:\ \Lambda_\mathcal{H}^\circ(\omega)\subset-\Lambda^\circ(\omega)\},$$
where $\Lambda_\mathcal{H}^\circ=\Lambda^\circ\cap\mathcal{H}_t$. For any $P\in\mathfrak{P}(\Omega_t)$, by \cite[(4.5)]{Nutz2}, there exists a Borel-measurable mapping $\Lambda_P^\circ:\ \Omega_t\twoheadrightarrow\mathbb{R}^d$ with non-empty closed values such that $\Lambda_P^\circ=\Lambda^\circ$ $P$-a.s.. This implies that the graph$(\Lambda_P^\circ)$ is Borel (see \cite[Theorem 18.6]{IDA}). Then it can be shown directly from the definition \eqref{e15} that $\Lambda_{\mathcal{H},P}^\circ:=\Lambda_P^\circ\cap\mathcal{H}_t$ is u.m. Thanks to the closedness of $-\Lambda^\circ$, the set
$$N_{t,P}^c=\{\omega:\ \Lambda_{\mathcal{H},P}^\circ(\omega)\subset-\Lambda^\circ(\omega)\}=\cap_{y\in\mathbb{Q}^d}\{\omega:\ \text{dist}(y,\Lambda_{\mathcal{H},P}^\circ(\omega))\geq\text{dist}(y,-\Lambda^\circ(\omega))\}$$
is u.m. Therefore, there exists a Borel measurable set $\tilde N_{t,P}^c$, such that $\tilde N_{t,P}^c=N_{t,P}^c=N_t^c\ P$-a.s. Thus $N_t^c$ is u.m. by \cite[Lemma 7.26]{Shreve}.

It remains to show that $N_t$ is $\mathcal{P}$-polar. If not, then there exists $P_*\in\mathcal{P}$ such that $P_*(N_t)>0$. Similar to the argument above, there exists a map $\Lambda_*^\circ:\ \Omega_t\twoheadrightarrow\mathbb{R}^d$ with a Borel measurable graph$(\Lambda_*^\circ)$, such that 
\begin{equation}\label{e71}
\Lambda_*^\circ=\Lambda^\circ\ P_*\text{-a.s.}.
\end{equation} Let
$$\Phi(\omega):=\{(y,P)\in(\Lambda_*^\circ\cap\mathcal{H}_t)(\omega)\times\mathcal{P}_t(\omega):\ E_P[y\Delta S_t(\omega,\cdot)]>0\},\quad\omega\in\Omega_t.$$
Then $N_t=\{\Phi\neq\emptyset\}\ P_*$-a.s. by \eqref{e17}, \eqref{e70} and \eqref{e71}. It is easy to see that (with a slight abuse of notation)
$$\text{graph}(\Phi)=[\text{graph}(\mathcal{P}_t)\times\mathbb{R}^d]\cap[\mathfrak{P}(\Omega)\times\text{graph}(\Lambda_*^\circ)]\cap\{E_P[y\Delta S_t(\omega,\cdot)]>0\}\cap[\mathfrak{P}(\Omega)\times\text{graph}(\mathcal{H}_t)]$$
is analytic. Therefore, by the Jankov-von Neumann Theorem \cite[Proposition 7.49]{Shreve}, there exists a u.m. selector $(y,P)$ such that $(y(\cdot),P(\cdot))\in\Phi(\cdot)$ on $\{\Phi\neq\emptyset\}$. As $N_t=\{\Phi\neq\emptyset\}\ P_*-a.s.$, $y$ is $P_*$-a.s. an arbitrage on $N_t$. Redefine $y=0$ on $\{y\notin\Lambda^\circ\cap\mathcal{H}_t\}$, and $P$ to be any u.m. selector of $\mathcal{P}_t$ on $\{\Phi=\emptyset\}$. (Here we redefine $y$ on $\{y\notin\Lambda^\circ\cap\mathcal{H}_t\}$ instead of $\{\Phi\neq\emptyset\}$ in order to make sure that $y(\cdot)\in\Lambda^\circ(\cdot)$ so that $y\Delta S_t\geq 0\ \cP-q.s.$.) So we have that $y(\cdot)\in\mathcal{H}_t(\cdot)$, $P(\cdot)\in\mathcal{P}_t(\cdot)$, $y\Delta S_t\geq 0\ \mathcal{P}-q.s.$, and
\begin{equation}\label{e16}
P(\omega)\{y(\omega)\Delta S_t(\omega,\cdot)>0\}>0\quad\text{for}\ P_*\text{-a.s.}\ \omega\in N_t.
\end{equation}
Now define $H=(H_0,\dotso,H_{T-1})\in\mathcal{H}$ satisfying 
$$H_t=y, \text{ and }H_s=0,\ s\neq t.$$
Also define 
$$P^*=P_*|_{\Omega_t}\otimes P\otimes P_{t+1}\otimes\dotso\otimes P_{T-1}\in\mathcal{P},$$
where $P_s$ is any u.m. selector of $\mathcal{P}_s$, $s=t+1,\dotso,T-1$. Then $(H\cdot S)_T\geq 0\ \mathcal{P}-q.s.$, and $P^*\{(H\cdot S)_T>0\}>0$ by \eqref{e16}, which contradicts NA$(\mathcal{P})$.
\end{proof}

\subsection{Proof of \leref{l15}}
\begin{proof}
Let
$$\Phi(\omega):=\{(R,\hat R)\in\mathfrak{P}(\Omega)\times\mathfrak{P}(\Omega):\ P(\omega)\ll R\ll \hat R\},\quad\omega\in\Omega_t,$$
which has an analytic graph as shown in the proof of \cite[Lemma 4.8]{Nutz2}. Consider $\Xi:\ \Omega_t\twoheadrightarrow\mathfrak{P}(\Omega)\times\mathfrak{P}(\Omega)$,
\begin{equation}\notag
\begin{split}
\Xi(\omega):=&\{(Q,\hat P)\in\mathfrak{P}(\Omega)\times\mathfrak{P}(\Omega):\ E_Q|\Delta S_t(\omega,\cdot)|<\infty,\ E_Q[y\Delta S_t(\omega,\cdot)]\leq 0,\ \forall y\in\mathcal{H}_t(\omega),\\
&P(\omega)\ll Q\ll \hat P\in\mathcal{P}_t(\omega)\}.
\end{split}
\end{equation}
Recall the analytic set $\Psi_{\mathcal{H}_t}$ defined \asref{a2}(iii). We have that
$$\text{graph}(\Xi)=[\Psi_{\mathcal{H}_t}\times\mathfrak{P}(\Omega)]\cap[\mathfrak{P}(\Omega)\times\text{graph}(\mathcal{P}_t)]\cap\text{graph}(\Phi)$$
is analytic. As a result, we can apply the Jankov-von Neumann Theorem \cite[Proposition 7.49]{Shreve} to find u.m. selectors $Q(\cdot),\hat P(\cdot)$ such that $(Q(\cdot),\hat P(\cdot))\in\Xi(\cdot)$ on $\{\Xi\neq\emptyset\}$. We set $Q(\cdot):=\hat P(\cdot):=P(\cdot)$ on $\{\Xi=\emptyset\}$. By \thref{t1}, if \asref{a2}(ii) and NA$(\mathcal{P}_t(\omega))$ hold, and $P(\omega)\in\mathcal{P}_t(\omega)$, then $\Xi(\omega)\neq\emptyset$. So our construction satisfies the conditions stated in the lemma.

It remains to show that graph$(\mathcal{Q}_t)$ is analytic. Using the same argument for $\Xi$, but omitting the lower bound $P(\cdot)$, we see that the map $\tilde\Xi:\ \Omega_t\Mapsto\mathfrak{P}(\Omega)\times\mathfrak{P}(\Omega)$,
\begin{equation}\notag
\begin{split}
\tilde\Xi(\omega):=&\{(Q,\hat P)\in\mathfrak{P}(\Omega)\times\mathfrak{P}(\Omega):\ E_Q|\Delta S_t(\omega,\cdot)|<\infty,\ E_Q[y\Delta S_t(\omega,\cdot)]\leq 0,\ \forall y\in\mathcal{H}_t(\omega),\\
&Q\ll \hat P\in\mathcal{P}_t(\omega)\}
\end{split}
\end{equation}
has an analytic graph. Since graph$(\mathcal{Q}_t)$ is the image of graph$(\tilde\Xi)$ under the canonical projection $\Omega_t\times\mathfrak{P}(\Omega)\times\mathfrak{P}(\Omega)\rightarrow\Omega_t\times\mathfrak{P}(\Omega)$, it is also analytic.
\end{proof}

\subsection{Proof of \leref{l9}}
\begin{proof}
Similar to the argument in \cite[Lemma 4.8]{Nutz2}, we can show that the set
$$J:=\{(P,Q)\in\mathfrak{P}(\Omega)\times\mathfrak{P}(\Omega):\ Q\ll P\}$$
is Borel measurable. Thus, for  $\Xi:\ \Omega_t\twoheadrightarrow\mathfrak{P}(\Omega)$
\begin{equation}\notag
\Xi(\omega)=\{Q\in\mathfrak{P}(\Omega):\ Q\lll\mathcal{P}_t(\omega)\},
\end{equation}
graph$(\Xi)$ is analytic since it is the projection of the analytic set 
$$[\Omega_t\times J]\cap[\text{graph}(\mathcal{P}_t)\times\mathfrak{P}(\Omega)]$$
onto $\Omega_t\times\mathfrak{P}(\Omega)$. By \asref{a4}(ii), the function $\hat A:\ \Omega_t\times\mathfrak{P}(\Omega)\mapsto\mathbb{R}^*$,
$$\hat A(\omega, Q)=A(\omega, Q)1_{\{E_Q|\Delta S_t(\omega,\cdot)|<\infty\}}+\infty 1_{\{E_Q|\Delta S_t(\omega,\cdot)|=\infty\}}$$
is l.s.a. As a result,
$$\text{graph}(\mathfrak{Q}_t)=\text{graph}(\Xi)\cap\{\hat A<\infty\}$$
is analytic.
\end{proof}

\subsection{Proof of \leref{l7}}
\begin{proof}
Denote the right side above by $\mathfrak{R}$. Let $R=Q_0\otimes\dotso\otimes Q_{T-1}\in\mathfrak{R}$. Without loss of generality, we can assume that $Q_t:\ \Omega_t\mapsto\mathfrak{P}(\Omega)$ is Borel measurable and $Q_t(\cdot)\in\mathfrak{Q}_t(\cdot)$ on $\{\mathfrak{Q}_t\neq\emptyset\}$ $Q^{t-1}:=Q_0\otimes\dotso\otimes Q_{t-1}$-a.s., $t=1,\dotso,T-1$. Let
$$\Phi_t(\omega):=\{(Q,P)\in\mathfrak{P}(\Omega)\times\mathfrak{P}(\Omega):\ Q_t(\omega)=Q\ll P\in\mathcal{P}_t(\omega)\},\quad\omega\in\Omega_t,\ t=0,\dotso, T-1.$$
Similar to the argument in the proof of \cite[Lemma 4.8]{Nutz2}, it can be shown that graph$(\Phi)$ is analytic, and thus there exists u.m. selectors $\hat Q_t(\cdot), \hat P_t(\cdot)$, such that $(\hat Q_t(\cdot),\hat P_t(\cdot))\in\Phi(\cdot)$ on $\{\Phi_t\neq\emptyset\}$. We shall show by an induction that for $t=0,\dotso,T-1$, 
$$\Phi_t\neq\emptyset\text{ for }t=0, \text{ and }\{\Phi_t=\emptyset\}\text{ is a }Q^{t-1}\text{-null set for }t=1,\dotso T-1,$$
and there exists a universally selector of $\mathcal{P}_t$ which we denote by $P_t(\cdot):\ \Omega_t\mapsto\mathfrak{P}(\Omega)$ such that
$$Q^t=\hat Q_0\otimes\dotso\otimes\hat Q_t\ll P_0\otimes\dotso\otimes P_t.$$
Then by setting $t=T-1$, we know $R=Q^{T-1}\in\mathfrak{Q}$. It is easy to see that the above holds for $t=0$. Assume it holds for $t=k<T-1$. Then $\{\Phi_{k+1}=\emptyset\}\subset\{Q_{k+1}(\cdot)\notin\mathfrak{Q}_{k+1}(\cdot)\}$ is a $Q^k$-null set by \leref{l6} and the induction hypothesis. As a result, $\hat Q_{k+1}=Q_{k+1}\ Q^k$-a.s., which implies that $Q^{k+1}=\hat Q_0\otimes\dotso\otimes\hat Q_{k+1}$. Setting $P_{k+1}=\hat P_{k+1} 1_{\{\Phi\neq\emptyset\}}+\tilde P_{k+1}1_{\{\Phi=\emptyset\}}$, where $\tilde P_{k+1}(\cdot)$ is any u.m. selector of $\mathcal{P}_{k+1}$, we have that $P_0\otimes\dotso\otimes P_{k+1}\in\mathcal{P}^{k+1}$. Since $Q_{k+1}(\omega)\ll P_{k+1}(\omega)$ for $Q^k$-a.s. $\omega\in\Omega_k$, together with the induction hypothesis, we have that $Q^{k+1}\ll P_0\otimes\dotso\otimes P_{k+1}$. Thus we finish the proof for the induction.

Conversely, for any $R\in\mathfrak{Q}$, we may write $R=Q_0\otimes\dotso\otimes Q_{T-1}$, where $Q_t:\ \Omega_t\mapsto\mathfrak{P}(\Omega)$ is some Borel kernel, $t=0,\dotso,T-1$. Then $Q_t(\omega)\in\mathfrak{Q}_t(\omega)$ for $Q^{t-1}$-a.s. $\omega\in\Omega_{t-1}$. Thanks to the analyticity of graph$(\mathfrak{Q}_t)$, we can modify $Q_t(\cdot)$ on a $Q^{t-1}$-null set, such that the modification $\hat Q_t(\cdot)$ is u.m. and $\hat  Q_t(\cdot)\in\mathfrak{Q}_t(\cdot)$ on $\{\mathfrak{Q}_t\neq\emptyset\}$. Using a forward induction of this modification, we have that $R=\hat Q_0\otimes\dotso\otimes\hat Q_{T-1}\in\mathfrak{R}$.
\end{proof}

\end{appendices}

\bibliographystyle{siam}
\bibliography{ref}
\end{document}